\documentclass[a4paper, 11pt]{amsart}
\usepackage{graphicx}
\usepackage[english]{babel}
\usepackage{amssymb,amsmath,amsthm}
\usepackage[margin=1in]{geometry}
\usepackage{hyperref}
\usepackage{bbm}
\usepackage{booktabs}
\usepackage{tikz-cd}
\usepackage{enumitem}
\usepackage{amscd}
\usepackage{mathtools}
\newcommand{\comment}[1]{}
\newtheorem{theorem}{Theorem}
\numberwithin{theorem}{section}

\newtheorem{definition}[theorem]{Definition}
\newtheorem{lemma}[theorem]{Lemma}

\newtheorem{example}[theorem]{Example}
\newtheorem{proposition}[theorem]{Proposition}
\newtheorem{corollary}[theorem]{Corollary}

\newtheorem{notation}[theorem]{Notation}

\newcommand{\aG}{\widehat{G}}
\newcommand{\cL}{\mathcal{L}}
\newcommand{\cA}{\mathcal{A}}
\newcommand{\cB}{\mathcal{B}}
\newcommand{\cC}{\mathcal{C}}

\def\Z{{\mathbb Z}}
\def\N{{\mathbb N}}

\def\R{{\mathbb R}}
\def\S{{\mathbb S}}

\def\sgn{{\operatorname{sgn}}}

\newcommand{\good}{triangle-regular }
\renewcommand{\phi}{\varphi}

\def\P{{\mathcal P}}
\def\F{{\mathcal F}}
\newcommand{\Aut}{\operatorname{Aut}}

\newcommand{\wt}{\operatorname{wt}}

\begin{document}

\title{Binomiality of colored Gaussian models}

\author{Benjamin Biaggi}
\address{Mathematical Institute, University of Bern, Alpeneggstrasse
22, 3012 Bern, Switzerland}
\email{benjamin.biaggi@unibe.ch}

\author{Jan Draisma}
\address{Mathematical Institute, University of Bern, Sidlerstrasse 5,
3012 Bern, Switzerland}
\email{jan.draisma@unibe.ch}

\author{Magdal\'ena Mi\v{s}inov\'a}
\address{Deparment of Mathematics and Statistics, University of Konstanz, Universitätsstrasse 10, 78464 Konstanz, Germany}
\email{magdalena.misinova@gmail.com}
\thanks{BB is funded by JD's Swiss National Science Foundation project
grant 200021-227864.}

\maketitle
\begin{abstract}
Following earlier work by Coons--Maraj--Misra--Sorea and Misra--Sullivant,
we study colored, undirected Gaussian graphical models, and present a 
necessary and sufficient condition for such a model to have binomial
vanishing ideal. These conditions involve {\em Jordan schemes}, a
variant of {\em association schemes},
well-known structures in algebraic combinatorics. Using association
schemes without transitive group action, we refute the conjecture by
Coons-Maraj-Misra-Sorea that binomiality implies that the color classes
must be orbits under the automorphism group of the colored graph.
\end{abstract}

\section{Introduction}

Throughout applications of algebraic geometry, and in particular in
algebraic statistics, {\em binomial ideals} play a prominent role. On
the one hand, the binomials in such an ideal may have a meaning for
the model under consideration---e.g., as Markov basis used in testing
whether the model fits the data \cite{Diaconis98}---and on the other hand,
binomial ideals are simply easier to compute with than general ideals. As
a consequence, much literature discusses conditions on algebraic models
that guarantee that their defining ideal is binomial, possibly after a
linear coordinate change. For instance, for group-based models in phylogenetics, this
line of research was initiated in \cite{Sturmfels05b}, which led to much
follow-up work on degree bounds (e.g.  \cite{Michalek12,Noren12}). And for
Gaussian graphical models, the topic of this paper, the study of
binomiality was initiated in \cite{MS21}, inspired by conjectures in
\cite{Sturmfels10}. In
\cite{Maraj23} a Lie-algebraic necessary condition was derived for the
existence of a linear coordinate change that makes an ideal binomial,
and \cite{Kahle24} presents a more general algorithm that tests whether an ideal
has this property.

In this paper, following \cite{MS21} and \cite{CMMPS23}, we study
binomiality of the vanishing ideal of colored, undirected, Gaussian
graphical models; we do not consider linear coordinate changes. Thus let $G$ be an undirected graph on the vertex
set $[n]:=\{1,\ldots,n\}$ and let $\lambda: \binom{[n]}{1} \sqcup
\binom{[n]}{2} \to \{1,\ldots,c\}$ be a map, called the coloring. By
abuse of notation, we will often write $G$ rather than $(G,\lambda)$ for
the colored graph. To this combinatorial data, we associate a linear
subspace $\cL_G$ of the space $\S^n$ of real symmetric matrices as follows:
\begin{align*} 
\cL_G:=&\{K=(k_{ij})_{i,j} \in \S^n \mid \forall i,j \in [n]: 
(k_{ij} \neq 0 \Rightarrow i=j \text{ or } \{i,j\} \text{ is an edge in }
G) \text{ and }\\
&\forall i,j,l,m \in [n]: (\lambda(\{i,j\})=\lambda(\{l,m\}) \Rightarrow k_{ij}=k_{lm})\}. 
\end{align*}
If $K \in \cL_G$ is positive definite, then $\Sigma=K^{-1}$ is the {\em
correlation matrix} of an $n$-tuple of jointly Gaussian random variables
with {\em concentration matrix} $K$.  This motivates the definition of
the semialgebraic set
\[ M_G:=\{K^{-1} \mid K \in \cL_G \text{ positive definite}\}, \]
the corresponding statistical model. Let $I_G=I_{G,\lambda} \subseteq
\R[\sigma_{11},\sigma_{12},\ldots,\sigma_{nn}]$ be the ideal of all
polynomials that vanish identically on $M_G$. Our main question is: under
what conditions on $(G,\lambda)$ is $I_{G}$ generated by {\em
binomials}? In this paper, a binomial is a polynomial of the form 
$\sigma^\alpha - \sigma^\beta$ for multi-indices $\alpha$ and $\beta$.

Several simplifications are in order. First, if $G$ is the disjoint union
of subgraphs $G_1$ and $G_2$, then the answer to our question is yes if
and only if it is yes for both $G_1$ and $G_2$. So we may assume that $G$
is connected. Furthermore, from a statistical point of view, it makes
sense to use distinct colors for vertices of $G$ and for edges of $G$,
i.e., to require that the restrictions of $\lambda$ to ${\binom{[n]}{1}}$
and $\binom{[n]}{2}$ have disjoint image. In this setting, our main
result is as follows.

\begin{theorem} \label{thm:Main}
Assume that $G$ is connected and that $\lambda(\{i\}) \neq
\lambda(\{j,l\})$ for all $i,j,l$ with $j \neq l$. Then $I_{G}$ is generated by
binomials if and only if $G$ is a block graph and triangle-regular.
\end{theorem}

We will give precise definitions of {\em block graph} and {\em
triangle-regular} below. When all colors are distinct, i.e.,
$\lambda(\{i,j\})=\lambda(\{l,m\})$ if and only if $\{i,j\}=\{l,m\}$,
then triangle regularity is automatic. In this setting, the implication
$\Leftarrow$ was proved in \cite{MS21}, and an explicit generating
set of quadratic binomials for $I_G$ was found \cite[Theorem 5]{MS21}.
The implication $\Leftarrow$ and explicit quadratic binomial generators
were then generalised to RCOP block graphs in \cite{CMMPS23}. The RCOP
condition says that the color classes are orbits under the automorphism
group of the colored graph $(G,\lambda)$; we will see below that this
implies triangle regularity. Our main contributions are finding the
combinatorial condition of triangle regularity and, armed with this,
the proofs of both implications $\Leftrightarrow$. As an application of
the arrow $\Leftarrow$ in Theorem~\ref{thm:Main}, we show how to disprove \cite[Conjecture 7.8]{CMMPS23}:

\begin{corollary}
There exist colorings $\lambda$ of the complete graph $G=K_n$ for which
$I_{G,\lambda}$ is binomial but the coloring is not RCOP.
\end{corollary}

\begin{proof}
Recall that a connected, undirected graph $H$ on $[n]$ is {\em strongly
regular} with parameters $(k,a,b)$ if it is $k$-regular (each vertex has
$k$ neighbors), any two adjacent vertices have $a$ common neighbors, and
any two non-adjacent vertices have $b$ common neighbors.  We now
color $G:=K_n$ by giving all vertices the color $1$, all edges in $H$
the color $2$, and all non-edges in $H$ the color $3$. The strong
regularity of $H$ implies that the colored graph $G$ 
satisfies triangle regularity, and $G=K_n$ is trivially a block
graph, hence by Theorem~\ref{thm:Main} the three-dimensional model $M$
has the property that $I_G$ is binomial.

Now take $H$ to be such that its automorphism group $\Aut(H)$ does
not act transitively on vertices, or not transitively on edges. Such
strongly regular graphs exist: the smallest (in terms of $n$) which is not
edge-transitive is the Shrikhande graph (see
Example~\ref{ex:Shrikhande}), and the smallest which are not
vertex-transitive have 25 vertices (see \cite{SRG} and the references
there). Now $\Aut(H)$ is also the automorphism group of the colored
graph $G$, and hence $G$ is not RCOP.
\end{proof}

The remainder of this paper is organised as follows: In
Section~\ref{sec:Definitions} we give all relevant definitions and derive
algebraic-combinatorial preliminary results. In Section~\ref{sec:Theorem}
we state a more complete version of Theorem~\ref{thm:Main}, namely,
Theorem~\ref{thm:CharToric}, which also gives explicit binomial generators
of $I_G$. In Section~\ref{sec:Necessary} we prove that triangle
regularity is a necessary condition for $I_G$ to be binomial. In
Section~\ref{sec:CombProp} we derive combinatorial properties of
triangle-regular graphs, which are then used in Section~\ref{sec:Gens}
to show that the binomial generators from Theorem~\ref{thm:CharToric} do
indeed generate $I_G$. It should be mentioned that those generators are the
same as those in \cite{CMMPS23}, which in turn are {\em mutatis mutandi}
those of \cite{MS21}, and much of our proof in Section~\ref{sec:Gens}
is inspired by their work. But we have found numerous combinatorial
shortcuts, which in particular avoid the use of transitive group actions.

\section{Definitions and preliminaries} \label{sec:Definitions}

\subsection{The vanishing ideal of a colored Gaussian model}

By a {\em graph} $G$, we will always mean a finite, undirected graph
without loops or multiple edges, 
and we write $V(G)$ for its set of vertices and $E(G) \subseteq
\binom{V(G)}{2}$ for its set of
edges. Assume, furthermore, that the vertices and edges of $G$ are
colored. We denote by $\lambda (i)$ the
color of a vertex and by $\lambda (\{i,j\})$ the color of an edge, and
we assume that edges and vertices have disjoint colors. Now $G$
defines a linear subspace $\mathcal{L}_G$ of the space 
$\S ^{V(G)}$ of symmetric $V(G) \times V(G)$-matrices, 
where the entries $k_{ij}$ of matrices $K \in \mathcal{L}_{G}$ must satisfy the following equations:
\begin{itemize}
    \item $k_{ij} = 0$ if $i \neq j$ and $\{i,j\}\not \in E(G)$ ,
    \item $k_{ii} = k_{jj}$ if $\lambda (i) = \lambda (j)$,
    \item $k_{ij} = k_{i'j'}$ if $\lambda (\{i,j\}) = \lambda (\{i',j'\})$.
\end{itemize}
The matrix $K $ is called concentration matrix. We define the set of symmetric matrices $\mathcal{L}_G^{-1}$ by
\[
\mathcal{L}_G ^{-1} = \overline{\{K^{-1} \mid K \in \mathcal{L}_G
\text{ invertible}\}},
\]
where the closure is the same when we work in the euclidean topology
as when we work in the Zariski topology. Note that
$\mathcal{L}_G^{-1}$ is the Zariski closure of the model $M_G$ from
the introduction. 

In this paper, we study the homogeneous ideal of polynomials
in 
\[\R [\Sigma] = \R [\sigma_{ij} \mid i,j \in
V(G)]/(\{\sigma_{ij}-\sigma_{ji} \mid i,j \in V(G)\}) 
\]
 that vanish on
$\mathcal{L}_G^{-1}$. We denote this ideal by $I_G$.  This ideal is the
kernel of the rational map
\[
  \rho _G : \R[\Sigma]  \to \R (K), \ 
  \sigma_{ij}  \mapsto \frac{(-1)^{i+j}K_{ij}}{\det (K)},  
\]
where $K_{ij} $ is the $(i,j)$ minor of $K$, and where $\R(K)$ is the
fraction field of the coordinate ring of $\cL_G$, i.e., of the
ring
\begin{align}
\R[(&k_{ij})_{i,j}]/(
\{k_{ij}-k_{ji} \mid i,j \in V(G)\} \cup 
\{k_{ij} \mid i \neq j \text{ and } \{i,j\} \text{ not an edge in
$G$}\} \cup \label{eq:RingK} \\
&\{k_{ii}-k_{jj} \mid \lambda(i)=\lambda(j)\} \cup
\{k_{ij}-k_{i'j'} \mid i \neq j, i' \neq j', \text{ and }
\lambda(\{i,j\})=\lambda(\{i',j'\})\}). \notag 
\end{align}

\subsection{Block graphs}

We start with a recursive definition of block graphs.

\begin{definition}
    A connected graph $G$ is called a {\em block graph} if it is either
    complete or else admits a partition $(A,B,\{c\}
    ) $ of its vertex set into three nonempty subsets such that $G$ has
    no edges connecting a vertex in $A$ to a vertex in $B$ and such that
    the graphs on $A \cup \{c\}$ and $B \cup \{c\}$ induced by $G$
    are themselves block graphs.
\end{definition}

A block graph is also called a {\em 1-clique sum of complete graphs}.
The vertex $c$ is said to {\em separate} $A$ and $B$.

\begin{definition}
A {\em path} in a graph $G$ is a connected, induced subgraph of $G$ without
cycles. 
\end{definition}

If $P$ is a path in $G$, then we can enumerate $V(P)$ as
$\{v_0,\ldots,v_k\}$ in such a manner that $E(P)=\{e_1,\ldots,e_k\}$
with $e_i=\{v_{i-1},v_i\}$. We then call $\ell(P):=k \geq 0$ the {\em
length} of $P$ and also write $(v_0,e_1,v_1,\ldots,v_{k-1},e_k,v_k)$ for
$P$. Note that in fact this also equals $(v_k,e_k,\ldots,e_1,v_0)$. We
call $v_0$ and $v_k$ the {\em endpoints} of $P$, and say that $P$ is a
path {\em between} or {\em connecting} $v_0$ and $v_k$.

\begin{proposition}[{\cite[Proposition~2]{MS21}}]
    Let $G$ be a connected block graph. Then for any two vertices $u$, $v$ in the graph, there is a unique shortest path in $G$, denoted $u \leftrightarrow v$, that connects $u$ to $v$.
\end{proposition}

\begin{definition}
    We call a path $P$ in a block graph $G$ a \emph{shortest path}, if there are $u,v\in V(G)$ such that $P$ is the shortest path between $u$ and $v$.
\end{definition}

The following lemma directly follows from the definition of block graphs:
\begin{lemma}\label{lemma:path-clique-edge}
    Let $P$ be a shortest path in a block graph $G$ and $C$ a maximal
    clique in $G$. Then $|E(P)\cap E(C)| \leq 1$. This implies that if
    $P = (v_0, e_1, v_1, \ldots , e_{k}, v_k)$ is a shortest path with
    $k \geq 2$ and $e_i \in E(C)$ for $i \neq 1,k$, then both
    $v_{i-1}$,$v_{i}$ are contained in at least one other maximal
    clique $C' \neq C$, and the same is true for $e_1$ and $v_1$ and
    $e_k$ and $v_{k-1}$. \hfill $\square$
\end{lemma}

\begin{lemma}{\cite[Lemma~4.4]{CMMPS23}}
\label{lemma:glueing_shortest_paths_same_edge}
    Let $G$ be a block graph, let $k$ be a nonnegative integer, and let $P = (v_0,e_1,v_1,
    \ldots, e_{k+1},v_{k+1})$ and $Q = (v_k,e_{k+1},v_{k+1}, \ldots, e_n,v_n)$ be two shortest paths that share the edge $e_{k+1}$. Then their union $(v_0,e_1,v_1,\ldots ,e_n,v_n)$ is the shortest path $v_0 \leftrightarrow v_n$.
\end{lemma}
\begin{corollary}
\label{Cor:shortest-path-criterion}
    Let $P$ and $Q$ be two shortest paths in a block graph $G$, both
    of length $\geq 1$ and both having a vertex $v$ as one endpoint. Denote $e$, resp. $f$, the edge incident to $v$ in $P$, resp. $Q$. If $e$ and $f$ are not contained in the same maximal clique, then $P\cup Q$ is again a shortest path.
\end{corollary}
\begin{proof}
    Let $e = \{u,v\}$ and $f= \{v,w\}$. As $e$ and $f$ are not
    contained in the same clique, the vertex set of $G$ admits a
    partition into $(A,B,\{v\})$ with $u \in A, w \in B$ and $v$
    separating $A$ and $B$. Then any path between $u$ and $w$ must
    contain $v$, and hence $(u,e,v,f,w)$ is a (the) shortest path
    between the vertices $u,w$. Now apply Lemma~\ref{lemma:glueing_shortest_paths_same_edge} to the three shortest paths $P, u\leftrightarrow w$ and $Q$.
\end{proof}

For vertices $u,v$, let $d(u,v)$ be the length of a shortest path between
$u$ and $v$. In our proof of Theorem~\ref{thm:Main} and its more precise
version Theorem~\ref{thm:CharToric}, we will recognize block graphs
using the following characterization.

\begin{proposition}\label{prop:char-of-block-graph}
    A graph $G$ is a block graph if and only if neither of the following two conditions is satisfied:
    \begin{enumerate}
        \item There exist $u,v\in V(G)$ between which there are 
	paths $P \neq Q$ of length $d(u,v)$ that are vertex-disjoint
	except for their endpoints $u,v$.

        \item There exist $u,v\in V(G)$ with $d(u,v)\geq 2$ connected
	both by a path $P$ of length $d(u,v)$ and by a path $Q$ of length 
	$d(u,v)+1$, where $P$ and $Q$ are vertex disjoint except for
	their endpoints $u,v$. 
    \end{enumerate}
\end{proposition}

Recall that there are two notions of $k$-connectivity for graphs: vertex
$k$-connectivity and edge $k$-connectivity. We consider only vertex
$k$-connectivity for $k=2$: we say that two vertices are 2-connected if
there exist two paths between them that are vertex disjoint except for the
endpoints. An inclusion-maximal subgraph of a graph $G$ such that every
two of its vertices are 2-connected is a 2-connected component of $G$.

\begin{proof}
    First, we note that $G$ is a block graph if and only if every 2-connected component is a clique.

    Vertices satisfying (1) or (2) need to be in the same 2-connected component and at distance at least two. Hence, if every 2-connected component is a clique, neither of the conditions can be satisfied.

    For the other implication, we assume that there are two
    2-connected vertices at distance at least two, and we need to find
    two (potentially different) vertices satisfying one of the
    conditions in the proposition.

    Consider a shortest cycle $C$ in $G$ such that the graph induced by
    $G$ on the vertices of $C$ is not a clique; there is some such cycle
    by the assumption. We claim that $C$ is either an induced cycle of
    length at least four or a four-cycle with one chord in $G$. Obviously,
    $C$ cannot have length three, so it has length at least four. If
    $C$ is an induced cycle, we are done. Otherwise, it has a chord
    $\{u,v\}$. Consider the two cycles $C_1,C_2$ that share the
    chord and together use all edges of $C$ exactly once. By minimality of $C$, $G$
    induces cliques on the vertex sets of the smaller cycles $C_1$ and $C_2$.  Since $G$
    does not induce a clique on the vertex set of $C$, we find vertices
    $x,y$ in $C_1,C_2$, respectively, with $\{x,y\} \not \in E(G)$. Then
    the graph induced on $\{x,v,y,u\}$ is a $4$-cycle with precisely
    one chord, and minimality of $C$ implies that $V(C)=\{x,v,y,u\}$,
    as desired.

    If $C$ is of length four, then we have vertices satisfying (1), no matter whether there is an additional chord or not. So we can assume that $C$ is an induced cycle.

    We will show that for any two vertices on $C$, there is no shorter
    path between them than the shorter arc of $C$ connecting them. The
    conclusion then follows immediately by taking two opposite or
    nearly opposite vertices on $C$, according as $C$ has even or odd
    length.

    Assume that there are two vertices $u,v\in C$ such that there is a
    shortcut between them: a path $S$ connecting $u,v$ that is shorter
    than the shortest arc $A$ on $C$ between $u,v$. Take $u$ and $v$
    with this property such that $\ell(A)$ is minimal. By minimality
    of $\ell(A)$, $S$ and $A$ are vertex disjoint except at their
    endpoints $u,v$. Hence, $A$ and $S$ together form a cycle $C'$
    shorter than $C$. The graph induced by $G$ on the vertices of $C'$
    is not a clique, since $u$ are not connected in $G$. This
    contradicts minimality of $C$.
\end{proof}

\subsection{Triangle-regular graphs}

\begin{definition}
    A colored graph $G$ is 
    \begin{itemize}
        \item \emph{vertex-regular} if any two vertices of the same
	color are incident with the same multiset of colored edges;
        \item \emph{edge-regular} if any two edges of the same color
	are incident with the same multiset of colored vertices;
        \item \emph{vertex triangle-regular} if any two vertices of
	the same color are incident with the same multiset of colored
	triangles;
        \item \emph{edge triangle-regular} if any two edges of the
	same color are incident with the same multiset of colored
	triangles; and 
        \item \emph{triangle-regular} if it is vertex-regular,
	edge-regular, and edge triangle-regular.
    \end{itemize}
\end{definition}

\begin{lemma} \label{lm:Triangle}
    A triangle-regular graph $G$ is vertex triangle-regular.
\end{lemma}
\begin{proof}
    Let $u$ be a vertex and let $\lbrace a,b,c\rbrace$ be a multiset of
    three edge colors.  By edge regularity, a triangle with those edge
    colors is completely determined as a colored triangle; we need to
    show that the number of such triangles incident to $u$ depends
    only on $\lambda(u)$. 
    
    Let $k_a$, $k_b$, and $k_c$ be the number of edges of color $a$,
    $b$, $c$ incident to $u$; by vertex regularity, these numbers
    depend only on $\lambda(u)$. Furthermore, let $t_a$, $t_b$, $t_c$
    be the number of triangles with edge colors $\lbrace a,b,c\rbrace$
    incident to an edge of color $a$, $b$, $c$, respectively.

    If $a$, $b$, $c$ are pairwise distinct, the number of triangles with
    edge colors $\lbrace a,b,c\rbrace$ incident to $u$ is $\frac{k_a\cdot
    t_a+k_b\cdot t_b+k_c\cdot t_c}{2}$. If $a \neq b=c$, then this
    number equals $\frac{k_a\cdot t_a + k_b\cdot
    t_b}{2}$. And if $a=b=c$, then this number equals $\frac{k_a\cdot t_a}{2}$.
\end{proof}

\subsection{Connections to other types of colored graphs}
In \cite{CMMPS23}, RCOP block graphs play a fundamental role. Denote
by $\Gamma (G)$ the group of all automorphisms of $G$ that preserve
the vertex and edge colors. Then $G$ is \emph{RCOP} if 
\begin{enumerate}
    \item the sets of vertex and edge colors are disjoint,
    \item for any vertices $u,v$ of the same color, there exists an
    $\gamma  \in \Gamma (G)$ with $\gamma (v) = u$,
    \item for any edges $e,f$ of the same color, there exists an
    $\gamma \in \Gamma (G)$ with $\gamma (f) = e$.
\end{enumerate}
Clearly, all RCOP graphs are triangle-regular. The converse is not true:
\begin{example} \label{ex:Shrikhande}
    The {\em Shrikhande graph} has vertex set $\Z _4 \times \Z _4$ and
    there is an edge $\{v,u\}$ if $u-v \in \{\pm (1,0), \pm (0,1), \pm
    (1,1),\}$. We define the colored complete graph $G_S$ on $\Z_4
    \times \Z_4$ by 
    \begin{align*}
        \lambda (u ) &= 0 \text{ for all vertices }u \in \Z_4 \times
	\Z_4 \\
        \lambda (\{u,v\}) &=1 \text{ if $\{u,v\}$ is an edge in the Shrikhande graph} \\
        \lambda (\{u,v\}) &=2 \text{ if $\{u,v\}$ is not an edge in the Shrikhande graph}.
    \end{align*}
    The Shrikhande graph is strongly regular, which implies that $G_S$
    is triangle-regular. The automorphism group of the Shrikhande graph
    acts transitively on vertices but not on edges, and this implies
    that $G_S$ is not RCOP.
\end{example}

Triangle-regular graphs are related to well-known combinatorial
structures called coherent configurations; we refer to \cite{Brouwer_coherent} and \cite{Cameron03_coherent_conf} for an introduction to the topic. A \emph{coherent configuration} on a finite set $X$ is a partition $\mathcal{R} = \{R_i \mid i \in I\}$ on $X \times X$ such that 
\begin{enumerate}
    \item There is a subset $H \subset I$ such that $\{R_h \mid h \in H\}$ is a partition of the diagonal $\{(x,x) \mid x \in X\}$
    \item For each $R_i$, the converse $\{(y,x) \mid (x,y) \in R_i\}$ is also in $\mathcal{R}$.
    \item For $i,j,k \in I$ and $(x,y) \in R_k$, the number of $z \in X$ such that $(x,z)\in R_i$ and $(z,y)\in R_j$ is a constant $p_{ij}^k$ that does not depend on the choice of $x,y.$
\end{enumerate}
A coherent configuration is an {\em association scheme} if the diagonal
is a single part in the partition and moreover each $R_i$ is equal to
its converse.

Given a coherent configuration $\mathcal{P}$, the \emph{symmetrization}
$\mathcal{P}^{\operatorname{sym}}$ is the partition of $X$ whose parts
are all unions of the parts of $\mathcal{P}$ and their converse. (This
symmetric partition of $X \times X$ does not need to be a coherent
configuration.) It defines a coloring of the complete graph in the obvious
manner, and it is easy to verify that this coloring is triangle-regular.

A \emph{Jordan scheme} is a symmetric partition $\mathcal{P}$ of $X \times X$ such that the diagonal is one element of the partition and the adjacency matrices ($(A_i)_{x,y} = 1 $ if $(x,y)\in R_i$) of the partition satisfies
\[
A_i A_j + A_j A_i = \sum _k q_{ij}^k A_k
\]
for some numbers $q_{ij}^k$. The name derives from the fact that the
linear span of the $A_i$ is then a Jordan algebra.
We will prove in Proposition~\ref{prop:completeGraphIsJordanAlgebra}
that also the adjacency matrices of triangle-regular graphs satisfy the
above, so all triangle-regular graphs with only one vertex color are
Jordan schemes. The symmetrization of a coherent configuration where
the diagonal is only one element of the partition is a Jordan scheme.
In \cite{KRM2019properjordanschemesexist}, it is showed that not all
Jordan schemes are symmetrizations of coherent configurations,
implying that not all triangle-regular graphs are symmetrizations of coherent configurations.
\begin{example}{\cite{KRM2019properjordanschemesexist}}
    Define the complete graph $J_{15}$ on $[15]$, where the coloring of an edge $(l,m)$ is defined by the $i,j$-th entry of the following symmetric matrix, and every vertex has the same color 0.
    \[
    \begin{pmatrix}
        0 & 1 & 1 & 2 & 3 & 4 & 2 & 3 & 4 & 2 & 3 & 4 & 2 & 3 & 4 \\
        1 & 0 & 1 & 3 & 4 & 2 & 3 & 4 & 2 & 3 & 4 & 2 & 3 & 4 & 2 \\
        1 & 1 & 0 & 4 & 2 & 3 & 4 & 2 & 3 & 4 & 2 & 3 & 4 & 2 & 3 \\     
        2 & 3 & 4 & 0 & 1 & 1 & 3 & 2 & 4 & 2 & 4 & 3 & 4 & 3 & 2 \\
        3 & 4 & 2 & 1 & 0 & 1 & 2 & 4 & 3 & 4 & 3 & 2 & 3 & 2 & 4 \\
        4 & 2 & 3 & 1 & 1 & 0 & 4 & 3 & 2 & 3 & 2 & 4 & 2 & 4 & 3 \\
        2 & 3 & 4 & 3 & 2 & 4 & 0 & 1 & 1 & 4 & 3 & 2 & 2 & 4 & 3 \\
        3 & 4 & 2 & 2 & 4 & 3 & 1 & 0 & 1 & 3 & 2 & 4 & 4 & 3 & 2 \\
        4 & 2 & 3 & 4 & 3 & 2 & 1 & 1 & 0 & 2 & 4 & 3 & 3 & 2 & 4 \\
        2 & 3 & 4 & 2 & 4 & 3 & 4 & 3 & 2 & 0 & 1 & 1 & 3 & 2 & 4 \\
        3 & 4 & 2 & 4 & 3 & 2 & 3 & 2 & 4 & 1 & 0 & 1 & 2 & 4 & 3 \\
        4 & 2 & 3 & 3 & 2 & 4 & 2 & 4 & 3 & 1 & 1 & 0 & 4 & 3 & 2 \\
        2 & 3 & 4 & 4 & 3 & 2 & 2 & 4 & 3 & 3 & 2 & 4 & 0 & 1 & 1 \\
        3 & 4 & 2 & 3 & 2 & 4 & 4 & 3 & 2 & 2 & 4 & 3 & 1 & 0 & 1 \\
        4 & 2 & 3 & 2 & 4 & 3 & 3 & 2 & 4 & 4 & 3 & 2 & 1 & 1 & 0
    \end{pmatrix}
    \]
    The corresponding partition on $[15] \times [15]$ is a Jordan scheme and the graph $J_{15}$ triangle-regular, but the partition is not a symmetrization of a coherent configuration.

    One can check by hand that this graph is triangle-regular. A
    symmetrization of a coherent configuration satisfies the following
    condition on the order of the triangle: If $(x,y)$ and $(u,v)$ are
    two edges with the same color, then the multiset of ordered tuples
    $\{(\lambda(\{x,z\}), \lambda(\{z,y\})) \mid z \in V(G)\}$ equals
    $\{(\lambda(\{u,w\}), \lambda(\{w,v\}) \mid w \in V(G)\}$ or
    $\{(\lambda(\{v,w\}), \lambda(\{w,u\})) \mid w \in V(G)\}$. This is not
    true for $J_{15}$. For example, the edges $(1,2)$ and $(4,5)$ have the same color, but the multisets of ordered pairs are $\{(1,1),(2,3),(2,3),(2,3),(2,3),(3,4),(3,4),(3,4),(3,4),(4,2),(4,2),(4,2),(4,2)\}$ and $\{(1,1),(2,3),(2,4),(2,4),(2,4),(3,2),(3,2),(3,2),(3,4),(4,2),(4,3),(4,3),(4,3)\} $.
\end{example}

\section{Statement of the main theorem} \label{sec:Theorem}

Given a colored block graph $G$ and two vertices $i,j$ of $G$, recall that
we write $i \leftrightarrow j$ for the unique shortest path between $i$
and $j$. We will also write $\Lambda(i \leftrightarrow j)$ for the multiset
of edge and vertex colors seen along this shortest path. With this
notation, we can state our more precise version of Theorem~\ref{thm:Main}.

\begin{theorem}[Characterization of graphs with binomial vanishing ideal]
\label{thm:CharToric}
    Let $G$ be a connected colored graph whose set of vertex colors is
    disjoint from its set of edge colors. The ideal $I_G$ is binomial
    if and only if $G$ is a triangle-regular block graph. Moreover,
    the set of generators of $I_G$ is the union of the linear
    generators
    $$
        \lbrace \sigma_{i,j}-\sigma_{k,l}\mid \Lambda(i\leftrightarrow j) = \Lambda(k\leftrightarrow l)\rbrace
    $$
    and the quadratic generators
    $$
        \lbrace \sigma_{i,j}\sigma_{k,l}-\sigma_{i,k}\sigma_{j,l}\mid
	E(i\leftrightarrow j)\cup  E(k\leftrightarrow l)=
	E(i\leftrightarrow k)\cup  E(j\leftrightarrow l)\rbrace,
    $$
    where the union of edges is taken as multisets. 
\end{theorem}

The rest of the paper is devoted to proving this theorem. 

\section{The only graphs with a binomial vanishing ideal are triangle
regular block graphs} \label{sec:Necessary}

The implication $\Rightarrow$ in Theorem~\ref{thm:CharToric} is proved by
contraposition. Here is the proof strategy: if $G$ is a connected colored
graph but not a block \good graph, then we will exhibit some configuration
in $G$, from which we construct a polynomial $p$ in $I_G$. To show
that $p$ is {\em not} in the ideal generated by the binomials in $I_G$,
we choose one of the monomials $\sigma^\alpha$ in $p$. Then we compute
$\rho_G(\sigma^\alpha)$ using the formula for $\rho_G(\sigma_{i,j})$
below, consider a few lowest degrees of edge-variables and show that
no other monomial in $p$ can have those degrees the same. Thus, we find
that $p$ violates the following easy lemma.

\begin{lemma}
Let $J$ be an ideal in a polynomial ring in variables
$x_1,\ldots,x_n$. Then $J$ is generated by binomials if and only if it
has the following property: for any nonzero $p \in J$ and any monomial
$x^\alpha$ with a nonzero coefficient in $p$, there exists another
monomial $x^\beta \neq x^\alpha$ with a nonzero coefficient in $p$,
such that $x^\alpha - x^\beta \in J$. \hfill $\square$
\end{lemma}

Recall that a walk in $G$ is an ordered sequence $w =
(v_0^we_1^wv_1^w\dots e_{k_w}^wv_{k_w}^w)$ where the $v_i$ are vertices
in $G$ and the $e_i=\{v_{i-1},v_i\}$ are edges in $G$; repetitions of
vertices and/or edges are allowed. Given two vertices $i$ and $j$,
let $W_{i,j}$ be the collection of all walks form $i$ to $j$.

We write $y_{\lambda(i)}$ for the image of $1-k_{ii}$ in the ring
in~\eqref{eq:RingK} and, for all edges $\{i,j\}$ of $G$, we write
$y_{\lambda\{i,j\}}$ for the image of $-k_{ij}$ in that ring. So that
ring may be identified with the polynomial ring in the $y$-variables.

\begin{proposition} \label{prop:Sigmaij}
    Let $G$ be a colored graph. Then we have 
\begin{equation}\label{eq:sigma-i-j}
    \rho_G(\sigma_{i,j}) = \sum_{w\in
    W_{i,j}}\prod_{\ell=0}^{k_w}\frac{1}{1-y_{\lambda(v_\ell^w)}}\prod_{\ell=1}^{k_w}
    y_{\lambda(e_\ell^w)}
\end{equation}
as a formal power series. 
\end{proposition}

By slight abuse of notation, we will in the following often write
$\sigma_{ij}$ instead of $\rho_G(\sigma_{ij})$.

\begin{proof}
    Let $K$ be some invertible matrix in $\mathcal L_G$. Write $K = I
    - \Psi$.  Then $\rho$ maps $\sigma_{ij}$ to the $(i,j)$-entry of 
    $$
        (I-\Psi)^{-1} = \sum_{k=0}^\infty \Psi^k,
    $$
    where the right-hand side is understood as a formal power series
    in the entries of $\Psi$.
    
    Fix $i,j\in [n]$. Our goal is to determine $\left((I-\Psi)^{-1}\right)_{i,j}$. Let $\mathcal I_{i,j,k}$ be the set of sequences $(i_0, i_1, \dots ,i_k)$ such that $i_0 = i$ and $i_k = j$. Then
    $$
        (\Psi^k)_{i,j} = \sum_{(i_0,\dots, i_k)\in \mathcal I_{i,j,k}}
	\prod_{\ell = 1}^{k} \Psi_{i_{\ell-1}, i_\ell} =
	\sum_{(i_0,\dots, i_k)\in \mathcal I_{i,j,k}} \prod_{\ell:
	\; i_{\ell-1} = i_\ell} y_{\lambda(i_\ell)} \cdot
	\prod_{\ell: \;  i_{\ell-1} \neq i_\ell}
	y_{\lambda(\{i_{\ell-1},i_\ell\})}
    $$
    If some pair $(i_{\ell-1},i_\ell)$ with $i_{\ell-1} \neq i_\ell$
    is not an edge in $G$, then the product is
    zero. Hence, each $(i_0,\dots, i_k)\in \mathcal I_{i,j,k}$ such
    that the product is non-zero has an underlying walk from the
    vertex $i$ to the vertex $j$ in $G$. Fix a walk $w =
    (v_0^we_1^wv_1^w\dots e_{k_w}^wv_{k_w}^w)$ such that $v_0 = i$,
    $v_{k_w} = j$. If we sum over all sequences $(i_0,\dots, i_k)$
    whose underlying path is $w$, we obtain
    $$
        \prod_{\ell=0}^{k_w}\left(\sum_{h=0}^\infty
	y_{\lambda(v_{\ell}^w)}^h\right)\cdot\prod_{\ell = 1}^{k_w}
	y_{\lambda(e_\ell^w)} = 
        \prod_{\ell=0}^{k_w} \frac{1}{1-y_{\lambda(v_\ell^w)}} 
\cdot\prod_{\ell = 1}^{k_w}
	y_{\lambda(e_\ell^w)} 
    $$
    To get formula \ref{eq:sigma-i-j}, we now sum over all walks from $i$ to $j$.
\end{proof}

In next few propositions, we make use of Theorem 2.5 from Talaska \cite{talaska2012}. We first recall the necessary definitions and describe the setting using the same notation.

Consider a {\em directed} graph $\aG$. Assign to each arrow $e\in
E(\aG)$ a
variable $x_e$ and assume that all those variables commute. We define
a \emph{weight} of a walk $w = (v_0^we_1^wv_1^w\dots
e_{k_w}^wv_{k_w}^w)$ as $\wt(w):= \prod_{i=1}^{k_w} x_{e_i^w}$. The
\emph{weighted path matrix} of $\aG$ is the matrix $M$ whose entry $m_{i,j}$ is $\sum_{w\in W_{i,j}} \wt(w)$.

We call a walk $w = (v_0^we_1^wv_1^w\dots e_{k_w}^wv_{k_w}^w)$ a
\emph{path} if $v_i^w\neq v_j^w$ for all $i\neq j$ and we call it a
\emph{cycle} if $v_0^w = v_{k_w}^w$ and $v_i^w\neq v_j^w$ for $i\neq
j$ otherwise. We identify a cycle $(v_0e_1\dots e_kv_0)$ with its
cyclic rotations such as $(v_1e_2\dots e_kv_0e_1 v_1)$.

Let $\mathcal C(\aG)$ denote the set of collections $\mathbf{C}$
of pairwise vertex-disjoint cycles. We require that each cycle in a
collection $\mathbf{C}$ contains at least one edge. We consider the
empty collection to be an element of $\mathcal{C}(\aG)$ with weight $1$.

Let $A=\{a_1,\dots ,a_k\}$ and $B= \{b_1,\dots ,b_k\}$ be two subsets of
$V(\aG)$, both of cardinality $k$, and let $\pi$ be an element
of the symmetric group $S_k$. Then we define $\mathcal
P_{A,B,\pi}(\aG)$ to be the set of collections $\mathbf{P} =
(P_1,\dots ,P_k)$ of pairwise vertex-disjoint paths where $P_i$ is a
path from $a_i$ to $b_{\pi(i)}$. Let $\mathcal{P}_{A,B}(\aG) =
\bigsqcup_{\pi \in S_k} \mathcal{P}_{A,B,\pi}(\aG)$. 

Let $\mathcal{F}_{A,B}(\aG)$ be the set of \emph{self-avoiding
flows} connecting $A$ to $B$, i.e. pairs $\mathbf{F} =
(\mathbf{P},\mathbf C)$ such that $\mathbf{P}\in
\mathcal{P}_{A,B}(\aG)$, $\mathbf{C}\in \mathcal C(\aG)$ and $\mathbf{P}$ and $\mathbf{C}$ are vertex-disjoint.

We define the sign of $\mathbf{P}\in \mathcal{P}_{A,B,\pi}(\aG)$
as $\sgn(\mathbf P) = \sgn (\pi)$ and the sign of $\mathbf{C}\in
\mathcal C(\aG)$ as $\sgn(\mathbf C) = -1^{|\mathbf C|}$, where
$|\mathbf C|$ is the number of cycles in $\mathbf C$. Finally, for
$\mathbf F\in \mathcal{F}_{A,B}(\aG)$ there is a unique decomposition of $\mathbf{F}$ as a collection of paths $\mathbf P$ and a collection of cycles $\mathbf C$ and we define $\sgn(\mathbf F) = \sgn(\mathbf{P})\sgn(\mathbf C)$.

\begin{theorem}{\cite[Theorem~2.5]{talaska2012}} \label{thm:Talaska-2012}
    Let $\aG$ be a directed graph with weighted path matrix $M$. Then the minor $\Delta_{A,B}(M)$, with rows indexed by $A$ and columns indexed by $B$, is given by
    $$
        \Delta_{A,B}(M) = \frac{\sum_{\mathbf{F}\in \F_{A,B}}\sgn
	(\mathbf{F})\wt (\mathbf{F})}{\sum_{\mathbf{C}\in
	\mathcal{C}(\aG)}\sgn (\mathbf{C})\wt(\mathbf{C})}.
    $$
\end{theorem}

Theorem~\ref{thm:Talaska-2012} is for directed graphs $\aG$,
while our graph $G$ in Theorem~\ref{thm:CharToric} is undirected. Also,
in Theorem~\ref{thm:Talaska-2012} we only have variables corresponding
to edges, while we have variables corresponding to both edges (more
precisely, edge colors) and vertices (more precisely, vertex colors).
To nevertheless be able to apply Talaska's theorem, we construct a
directed graph $\aG$ from $G$ by replacing each edge $\{u,v\}$ of $G$ 
with two directed edges $(u,v)$ and $(v,u)$, and we specialize Talaska's variables
$x_{(u,v)}$ and $x_{(v,u)}$ to $z_{\lambda(u)} y_{\lambda(\{u,v\})}$ and
$z_{\lambda(v)} y_{\lambda(\{u,v\})}$, respectively, where $z_{\lambda(i)}$
stands for $\frac{1}{1-y_{\lambda(i)}}$. Then the expression for
$\rho_G(\sigma_{ij})$ in Proposition~\ref{prop:Sigmaij} equals
$z_{\lambda(j)}$ times the $(i,j)$ entry of the weighted path matrix $M$ for 
$\aG$. Consequently, we have the fundamental identity 
\[ \rho_G(\det\Sigma_{A,B}) = \Delta_{A,B}(M)\cdot \prod_{v\in B}
z_{\lambda(v)}. \]
Note that in our setting, every non-trivial walk from some vertex to itself has at least two edges.

\begin{proposition} \label{prop:Block}
    If $G$ is not a block graph, then $I_G$ is not binomial.
\end{proposition}

In the proof below, we write $N(u)$ for the neighbors in $G$ of a
vertex $u$.

\begin{proof}
    If graph is not a block graph, there must be two vertices $u$ and $v$ satisfying one of the conditions of Proposition \ref{prop:char-of-block-graph}.

    Case 1: There are $u,v$ satisfying condition (1). See
    Figure~\ref{fig:case1} for an illustration. Take $u,v$ with the
    smallest possible $d(u,v)$. Let $A = N(u)\cup \lbrace u\rbrace$,
    $B=N(u) \cup \lbrace v\rbrace$ and consider
    $\rho_G(\det\Sigma_{A,B})$.
    This polynomial vanishes, since the set $P_{A,B}$ is empty: any
    path from $u$ to a vertex in $B$ intersects a path starting at a
    neighbor of $u$. Hence in Theorem~\ref{thm:Talaska-2012} we sum over an empty set, obtaining zero.

    Consider the monomial $m=\sigma_{u,v}\cdot \prod_{i \in N(u)}
    \sigma_{i,i}$ in $\det(\Sigma_{A,B})$ and consider the edge-variable degrees of the monomials in
    $\rho_G(m)$, using the formula~\eqref{eq:sigma-i-j}. The smallest
    non-zero edge-degree is $d(u,v)$, where we take trivial paths from
    every factor, except for $\sigma_{u,v}$. Hence the part of
    edge-degree $d(u,v)$ in $\rho_G(m)$ equals 
    $$
        \sum_{P\text{ shortest path between $u$ and $v$}}\left( \prod_{e\in E(P)} y_{\lambda(e)}\cdot \prod_{v\in V(P)} z_{\lambda(v)}\right).
    $$
    If some other monomial $m'$ in $\det(\Sigma_{A,B})$ has the property
    that $m'-m \in I_G=\ker(\rho_G)$, then $\rho_G(m')$ in particular
    has a nonzero part of edge-degree $d(u,v)$. We claim that $m'$
    must then be of the form $m_{i}:=\sigma_{u,i}\sigma_{v,i}\cdot
    \prod_{j\in N(u),\, j\neq i}\sigma_{j,j}$, where $i\in N(u)$
    lies on some shortest path from $u$ to $v$. Indeed, $m'$ must
    contain a factor $\sigma_{u,w}$ with $w\neq u$, which contributes
    at least one edge variable to $\rho(m')$. If $w=v$, all the other
    factors need to be $\sigma_{j,j}$ for $j\in N(u)$, i.e., $m'=m$
    again. Otherwise, $m'$ has a factor $\sigma_{v,w'}$ for some $w'
    \in N(u)$, contributing at least $d(u,v)-1$ edge variables. Hence,
    all the other factors must be $\sigma_{j,j}$. But if $w\neq w'$,
    there still needs to be factor $\sigma_{w,w''}$ for some $w''$. So
    really $i:=w=w'$ and $d(i,v) = d(u,v)-1$.

    By minimality of $d(u,v)$, the shortest path from $i$ to $v$
    is unique. Consequently, $\rho(m_{i})$
    has only one term of edge-degree $d(u,v)$, while $\rho(m)$ has at
    least two. Hence, there is no monomial $m'$ in $\det(\Sigma_{A,B})$
    such that $m-m' \in I_G$. Hence $I_G$ is not binomial.

    \begin{figure}[h]
        \centering
        \includegraphics[width=0.5\linewidth]{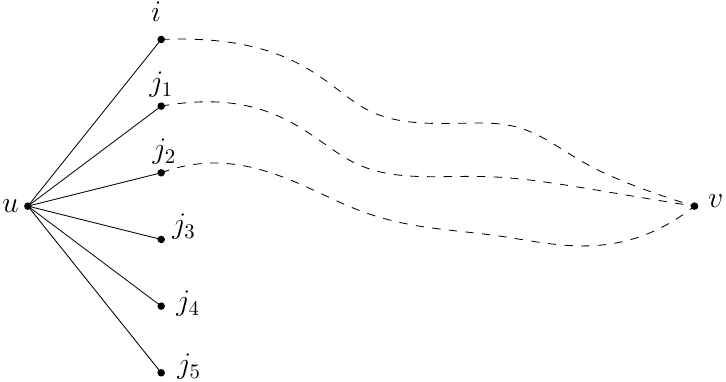}
        \caption{The dashed paths are the shortest paths between $u$ and $v$.}
        \label{fig:case1}
    \end{figure}

    Case 2: There are no two vertices satisfying condition (1), but
    there are two vertices satisfying condition (2). Take such
    vertices $u$ and $v$ with the smallest $d(u,v)$. See
    Figure~\ref{fig:case2}

    \begin{figure}[h]
        \centering
        \includegraphics[width=0.5\linewidth]{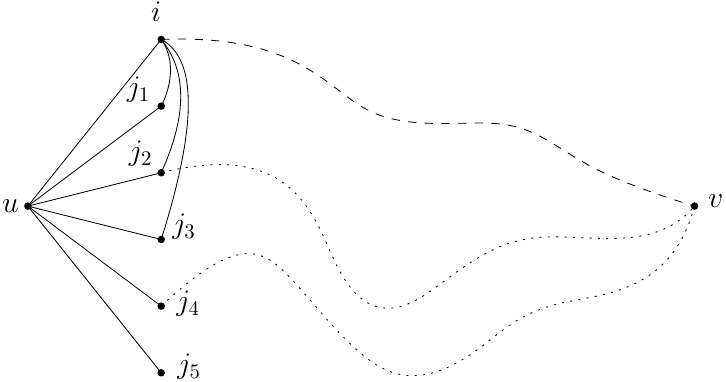}
        \caption{The dashed path is the unique shortest path, dotted paths are the possible paths of length $d(u,v)+1$.}
        \label{fig:case2}
    \end{figure}
    
    Consider the same determinant and the monomial $m$ as in
    the first case. Let $i \in N(u)$ be the vertex on the unique
    shortest path from $u$ to $v$. By the same argument as above,
    the only candidate monomial $m'$ with $m-m' \in \ker(\rho_G)$
    is $m_{i}=\sigma_{u,i}\sigma_{i,v}\cdot \prod_{j\in N(u),
    j\neq i}\sigma_{j,j}$. This time, the components in $\rho(m)$ and
    $\rho(m_{i})$ of lowest edge-degree are the same. Consider the second
    lowest one, which is $d(u,v)+1$. The factors $\sigma_{j,j}$ still need
    to contribute trivially, since otherwise they contribute by at least
    two edge variables.  The factor $\sigma_{i,v}$ can contribute only
    once with $d(u,v)-1$ edges as before, by the minimality of $d(u,v)$
    and the existence of the unique shortest path. So the $+1$ must come
    from $\sigma_{u,i}$. There are as many monomials with edge degree
    two in $\rho_G(\sigma_{u,i})$ as there are edges $(i,j)$ with $j \in
    N(u) \setminus \{i\}$. But each monomial of edge degree $d(u,v)+1$
    obtained like this appears also in $\rho_G(m)$. Moreover,
    there is in $\rho_G(m)$ one monomial more, namely the one
    corresponding to the path of length $d(u,v)+1$, which is vertex
    disjoint from the shortest path of length $d(u,v)$ except at
    endpoints. So $\rho_G(m)$ and $\rho_G(m_{i})$ are different
    and $I_G$ is not binomial. 
\end{proof}

\begin{proposition} \label{prop:Vertex}
    If the block graph $G$ is not vertex-regular, then $I_G$ is not binomial.
\end{proposition}
\begin{proof}
    Denote the two vertices violating the vertex regularity $u$ and $v$. 
    
    First assume $d(u,v)\geq 2$. Let $A = \lbrace u\rbrace\cup N(u)\cup N(v)$, $B = \lbrace v\rbrace\cup N(u)\cup N(v)$ and consider $\det\Sigma_{A,A}-\det\Sigma_{B,B}$.
    We use Theorem \ref{thm:Talaska-2012} to show that this polynomial
    lies in $\ker(\rho_G)$. Denote by $p$ the product 
    $$
        p := z_{\lambda(u)}\cdot\prod_{u_i\in N(u)}z_{\lambda(u_i)}\cdot \prod_{v_i\in N(v)}z_{\lambda(v_i)}=
        z_{\lambda(v)}\cdot \prod_{u_i\in N(u)}z_{\lambda(u_i)}\cdot\prod_{v_i\in N(v)}z_{\lambda(v_i)}.
    $$
    Note that every element of $\P_{A,A}$, resp. $\P_{B,B}$, must contain the trivial path from $u$ to $u$, resp. from $v$ to $v$. If we exchange the trivial path from $u$ to $u$ for the one from $v$ to $v$, we obtain $\F_{B,B}$ instead of $\F_{A,A}$. Hence by \ref{thm:Talaska-2012}, $\Delta_{A,A} = \Delta_{B,B}$. So,
    $$
        \rho_G(\det \Sigma_{A,A} - \det\Sigma_{B,B}) = p\cdot (\Delta_{A,A}-\Delta_{B,B}) = 0.
    $$

    Choose the monomial corresponding to the identity permutation in
    $\det\Sigma_{A,A}$ and denote it by $m_u$. In $\rho_G(m_u)$, the
    part of edge-degree zero is non-zero. Hence, the only monomial $m'$
    in our polynomial such that $\rho_G(m-m')$ might be zero is $m_v$,
    the monomial corresponding to the identity in $\det\Sigma_{B,B}$.

    Let us now compare a few smallest edge-degrees of monomials in
    $\rho_G(m_u)$ and $\rho_G(m_v)$. First, the components of 
    edge-degree zero are indeed the same, namely, $p$. Second, neither 
    $\rho_G(m_u)$ nor $\rho_G(m_v)$ contain monomials of edge-degree
    one. 

    In edge-degree two, all the terms in $\rho_G(m_u)$ are of the form
    $p\cdot z_{\lambda(i)}z_{\lambda(j)} y_{\lambda (e)}^2$, where
    $e=\{i,j\}$ is an edge with $i\in A$. Similarly for $\rho(m_v)$
    and $B$. Hence after canceling the terms that appear in both, we
    see that the edge-degree two part of
    $\rho_G(m_u - m_v)$ equals 
    $$
        p\cdot\left(z_{\lambda(u)}\cdot\sum_{u_i\in N(u)}
	z_{\lambda(u_i)}\cdot y_{\lambda(\{u,u_i\})}^2 -
	z_{\lambda(v)}\cdot \sum_{v_i\in N(v)} z_{\lambda(v_i)}\cdot
	y_{\lambda(\{v,v_i\})}^2\right).
    $$ 
    If the multiset of colors of edges adjacent to $u$ and $v$ are
    different, then this is a non-zero polynomial, so $I_G$ cannot be
    binomial.

    If $d(u,v) = 1$ we proceed similarly. We just consider $A = (N(u)\cup N(v))\setminus \lbrace v\rbrace$, $B = (N(u)\cup N(v))\setminus \lbrace u\rbrace$ and $p:= \prod_{w\in A}z_{\lambda (w)} = \prod_{w\in B}z_{\lambda (w)}$.
\end{proof}

\begin{proposition}\label{prop:edge-transitivity}
    If the block graph $G$ is vertex-regular but not edge-regular or not edge
    triangle-regular, then $I_G$ is not binomial.
\end{proposition}
\begin{proof}
    Consider two edges $\{u_1,v_1\}$ and $\{u_2,v_2\}$ of the same color. Let $A = \lbrace u_1, v_1, u_2,v_2\rbrace \cup N(u_1)\cup N(v_1)\cup N(u_2)\cup N(v_2)$ 
    and consider $\det\Sigma_{A\setminus\lbrace u_1\rbrace, A\setminus\lbrace v_1\rbrace} - \det\Sigma_{A\setminus\lbrace u_2\rbrace,A\setminus\lbrace v_2\rbrace}$.

    This polynomial lies in $\ker(\rho_G)$: Elements of $\P_{A\setminus\lbrace u_1\rbrace,A\setminus\lbrace v_1\rbrace}$ must contain the path $(v_1,\{v_1,u_1\},u_1)$ and the trivial paths corresponding to vertices $u_2$ and $v_2$. Similarly for $\P_{A\setminus\lbrace u_2\rbrace,A\setminus\lbrace v_2\rbrace}$. Except for those paths, elements of $\F_{A\setminus\lbrace u_1\rbrace,A\setminus\lbrace v_1\rbrace}$ and $\F_{A\setminus\lbrace u_2\rbrace,A\setminus\lbrace v_2\rbrace}$ are the same. Using Theorem \ref{thm:Talaska-2012}, the fact that $\lambda(\{u_1,v_1\}) = \lambda(\{u_2,v_2\})$ and factoring out $z_{\lambda(u_1)}z_{\lambda(v_1)}z_{\lambda(u_2)}z_{\lambda
    (v_2)}y_{\lambda(u_1,v_1)}$ from both sides, we obtain
    $$
        \rho_G(\det\Sigma_{A\setminus\lbrace u_1\rbrace,
	A\setminus\lbrace v_1\rbrace}) =
	\rho_G(\det\Sigma_{A\setminus\lbrace
	u_2\rbrace,A\setminus\lbrace v_2\rbrace}).
    $$
    Consider the monomial $m_1 = \sigma_{v_1,u_1}\cdot \prod_{v\in
    A\setminus \lbrace u_1, v_1\rbrace} \sigma_{v,v}$.
    Using~\eqref{eq:sigma-i-j}, we see that $\rho(m_1)$ has a non-zero
    component of edge-degree one. The only other such monomial in our
    difference of two determinants is $m_2 = \sigma_{v_2,u_2}\cdot \prod_{v\in A\setminus \lbrace u_2, v_2\rbrace} \sigma_{v,v}$.

    Denote by $q$ the product $\prod_{v\in A} z_{\lambda(v)}$ and
    consider the edge-degree two monomials of $\rho_G(m_1)$ and
    $\rho_G(m_2)$. For $i =1,2$, all the
    terms in $m_i$ are of the form $q\cdot z_{\lambda(x)}
    y_{\lambda(\{v_i,x\})} y_{\lambda(\{x,u_i\})}$, where $x\in N(u_i)\cap
    N(v_i)$. This shows that if the two edges $\{u_1,v_1\}$ and
    $\{u_2,v_2\}$ are adjacent to
    different multisets of colored triangles, then $I_G$ is not binomial. In
    particular, if we denote by $C_i$ the set of vertices in the
    maximal clique containing $\{u_i,v_i\}$, then
    $\lambda(C_1\setminus\lbrace u_1,v_1\rbrace) =
    \lambda(C_2\setminus\lbrace u_2,v_2\rbrace)$ as multisets,
    otherwise $m_1\neq m_2$ and the ideal is not binomial.

    Consider now edge-degree three. There are three types of path
    systems leading to the following terms in $\rho_G(m_i)$:
    \begin{enumerate}
        \item $q\cdot y_{\lambda(\{u_i,v_i\})}\cdot
	z_{\lambda(x)}z_{\lambda(y)}y_{\lambda(\{x,y\})}^2$ for some $x\in A$, $y\in N(x)$,
        \item $q\cdot
	z_{\lambda(x)}z_{\lambda(y)}y_{\lambda(\{v_i,x\})}y_{\lambda(\{x,y\})}
	y_{\lambda(\{y,u_i\})}$ for some $x,y\in C_i\setminus \lbrace u_i,v_i\rbrace$,
        \item $q\cdot
	z_{\lambda(u_i)}z_{\lambda(v_i)}y_{\lambda(\{u_i,v_i\})}^3$.
    \end{enumerate}

    The terms of the first type are the same for $i=1,2$. Regarding terms
    of type (2): since $\lambda(C_1\setminus\lbrace u_1,v_1\rbrace) =
    \lambda(C_2\setminus\lbrace u_2,v_2\rbrace)$ as multisets, the number
    of terms of type (2) with fixed $z_{\lambda(x)}z_{\lambda(y)}$ is the
    same in $m_1$ and $m_2$. This is in particular true for those terms
    with $\{\lambda(x),\lambda(y)\}=\{\lambda(u_i),\lambda(v_i)\}$
    as multisets, which might yield the same monomials
    as in (3).  Hence if $\rho(m_1)$ equals $\rho(m_2)$,
    then the terms of type (3) must agree. This implies that
    $\{\lambda(u_1),\lambda(v_1)\}=\{\lambda(u_2),\lambda(u_2)\}$. If
    this does not hold, then $I_G$ is not binomial.

    Note that the proof works even if some of the pairs of vertices $\{u_1, u_2\}$, $\{u_1,v_2\}$, $\{v_1,u_2\}$, $\{v_1,v_2\}$ are neighbors or even the same vertex.
\end{proof}

Note that in the preceding two proofs we could consider larger sets $A$ and $B$. For example in Proposition~\ref{prop:edge-transitivity}, the proof works all the same with $A = V(G)$.

\begin{corollary}
    Let $G$ be a colored graph such that the sets of colors of
    vertices and of edges are disjoint. If $I_G$ is binomial, $G$ is a triangle-regular block graph.
\end{corollary}

\begin{proof}
This follows from
Propositions~\ref{prop:Block},~\ref{prop:Vertex},~\ref{prop:edge-transitivity},
and Lemma~\ref{lm:Triangle}.
\end{proof}

\section{Combinatorial properties of triangle-regular block graphs}
\label{sec:CombProp}

We start by illustrating the main idea of this section in the following example:
\begin{example}
    Let $G$ be the triangle-regular block graph defined by the
    following picture, where the vertex colors are $1, \ldots ,7$ and
    the edge colors are as indicated in the picture: 

    \begin{center}
        \includegraphics[width=0.5\textwidth]{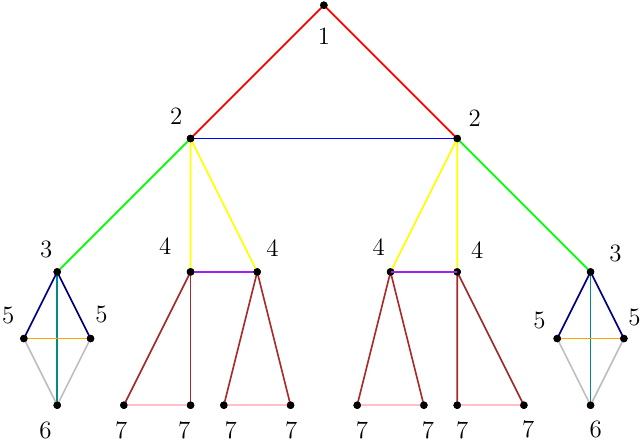}
    \end{center}

    Then we observe the following: two maximal cliques either have
    different edge and vertex colors (except for the vertices
    contained in multiple maximal cliques), or they have the same colors. 
    Arranging the graph as above, maximal cliques with the same colors
    lie at the same ``level''. 

    All shortest path only go ``up'' and then ``down'' in the picture.
    We also see that shortest paths $u \leftrightarrow v$ with
    $\lambda (u) = \lambda (v)$, are palindromic in terms of colors.
    Furthermore, if two shortest paths have the same multiset of edge
    and vertex colors, then they are isomorphic as colored paths. 
    
    The following picture describes the underlying colored rooted tree, where vertices of the same color correspond to maximal cliques with the same colors. Edges with the same color correspond to the vertices of the original graph of the same color connecting two or more maximal cliques:

    \begin{center}
        \includegraphics[width=0.4\textwidth]{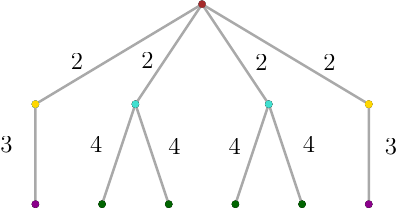}
    \end{center}
    
    The tree is very symmetric; in particular, for any two vertices of
    the same color, it has a color-preserving automorphism mapping one
    to the other. 
\end{example}

In the following, we will formalize these observations and show that
they hold for all triangle-regular block graphs.

\begin{proposition}\label{prop:removing-cliques}
    Given a connected, triangle-regular block graph $G$ with at least two maximal
    cliques, we can remove a set $\mathcal C$ of cliques from $G$ so that the
    following conditions hold. Here, by removing we mean deleting all
    the vertices contained {\em only} in cliques in $\mathcal{C}$. 
    \begin{enumerate}
        \item We obtain a connected triangle-regular block graph.
        \item All the cliques in $\mathcal{C}$ have the same multiset
	of colored vertices and colored edges.
        \item If $v\in C\in \mathcal{C}$ is a vertex, and $v' \in V(G)$ is a
	vertex of the same color: $\lambda(v')=\lambda(v)$, then $v'\in V(C')$ for some $C'\in \mathcal{C}$, similarly for edges.
        \item For each $C\in \mathcal{C}$ there is exactly one vertex
	$v_C\in V(C)$ contained in at least two maximal cliques. The
	color $\lambda(v_C)$  is independent of the choice of $C$, and
	the only vertices of $G$ with this color are the vertices
	$v_C$ with $C$ ranging through $\cC$.
        \item For any vertex $v\in V(C)\setminus \lbrace v_C\rbrace$,
	$\lambda(\{v,v_C\})$ depends only on $\lambda(v)$.
        \item Let $C\in \mathcal C$ and let $v_{C'}$ be a vertex with $\lambda(v_{C'}) = \lambda(v_C)$. Then the number of cliques in $\mathcal C$ containing $v_{C'}$ is the same as for $v_C$.
    \end{enumerate}
\end{proposition}
\begin{proof}
    We note that if a clique with only one vertex is connected to
    another vertex with an edge, then this forms another clique. Hence
    if there exists at least two maximal cliques, each must contain at least two vertices. 

    By the construction of a block graph, there is a maximal clique
    $C_0$ that shares exactly one vertex $v_{C_0}$ with some other
    maximal clique(s). There has to be some vertex $u_0\neq v_{C_0}$
    in $C_0$. We claim that 
    \[ \mathcal{C} = \lbrace C\text{ maximal
    clique}\mid \exists u\in V(C):\, \lambda(u) = \lambda(u_0)
    \rbrace \] is a collection as desired. 

    Fix some $C\in \mathcal{C}$ and $u\in C$ with $\lambda(u) =
    \lambda(u_0)$. Denote by $N(u_0)$ the set of neighbors of $u_0$
    and set $n := |N(u_0)|$. Since all the neighbors of $u_0$ are
    contained in $C_0$, $u_0$ is contained in $\binom{n}{2}$
    triangles. By triangle regularity of $G$, $u$ also has $n$
    neighbors and is contained in $\binom{n}{2}$ triangles. This is
    possible only if all neighbors of $u$ are contained in $C$, so $u$
    is contained in a unique maximal clique. By edge and vertex regularity, we also know the multiset of colors of neighbors of $u$. This proves item (2) for vertices. 

    Item (3) for vertices follows from the definition of $\mathcal{C}$
    and triangle regularity of $G$: If $v_0$ is any other vertex in $C_0$, and $v$ is a vertex with $\lambda(v) = \lambda(v_0)$, by triangle regularity $v$ has a vertex $u$ of color $\lambda(u_0)$ as a neighbor. Since $u$ is contained in a unique maximal clique, $v$ must be in it too.
    
    By vertex regularity, $\lambda (v_{C_0})$ appears exactly once
    in the multiset of colors of vertices in $C_0$ ($|N(v_{C_0})|
    > |V(C_0)|-1$, while all other vertices in $C_0$ have exactly
    $|V(C_0)|-1$ neighbors). For $C \in \cC$ let $v_C\in V(C)$ be the
    unique vertex of color $\lambda(v_{C_0})$. It has more than $|C|-1$
    edges, hence it is contained in another clique. Repeating the argument
    in the previous paragraph for all the other vertices in $C_0$, we
    see that $v_C$ is the only vertex of $C$ which can be contained in
    some other clique. This yields all but the last statement in item
    (4). For the last statement, note that by item (3) for vertices,
    any vertex $v$ with $\lambda(v)=\lambda(v_C)$ lies in some $C' \in
    \cC$. If $v \neq v_{C'}$, then $v$ has fewer neighbors than
    $v_{C'}$, and hence by vertex regularity cannot be of the same color. 

    Item (5) follows from edge and vertex regularity: consider
    vertices $v_0,v\in V(C) \setminus \{v_C\}$ with $\lambda(v) =
    \lambda(v_0)$. Then there is an edge of color
    $\lambda(\{v_0,v_{C}\})$ incident to $v$. Since $v_C$ is the unique vertex of color $\lambda(v_{C})$ in $C$, the other endpoint of that edge is indeed $v_C$.

    Now we prove item (2) for edges: Consider $C\in \mathcal C$ and an
    edge color $a$ appearing in $E(C)$. By item (2) for vertices and
    item (5), the multiset of colors of edges in $C$ containing $v_C$
    does not depend on the choice of $C$. For $v\in V(C) \setminus
    \{v_C\}$, let $M_a(v)$ be the number of edges of color $a$
    incident to $v$ in the clique $C'$ induced by $C$ on $V(C) \setminus
    \{v_C\}$. Then the number of edges of color $a$ in $C'$ is $\frac
    1 2\cdot \sum_{v\in V(C) \setminus \{v\}} M_a(v)$. Since $M_a(v)$
    depends only on $\lambda (v)$ and the multiset of colors of
    vertices in $C'$ does not depend on the choice of $C$, $\frac 1 2\cdot \sum_{v\in V(C)} M_a(v)$ does not depend on it either.
    
    Item (3) for edges: Consider an edge $e\in E(G)$ and $C \in \cC$
    with an edge $f=\{u,v\} \in E(C)$ such that $\lambda (e) = \lambda
    (f)$. Without loss of generality, $u$ is not equal to $v_C$,
    and hence $u$ has fewer neighbors than $v_C$ by item (4). By
    edge regularity, $e=\{u',v'\}$ where $\lambda(u)=\lambda(u')$, and
    by vertex regularity the number of neighbors of $u'$
    is equal to that of $u$. By (3) for vertices, $u'$ lies 
    in a maximal clique $C' \in \cC$, respectively, and since $u'$
    has fewer neighbors than $v_{C'}$, we find that $e$ also lies in
    $C'$. 

    Item (6): By item (3) for edges and by vertex regularity, the number
    of edges incident to $v_{C'}$ and contained in a clique in $\mathcal C$ does not depend on the choice of $v_{C'}$ among
    the vertices of color $\lambda(v_C)$. By item (2) for vertices, each
    clique in $\cC$ has the same number of vertices. So the claim follows.

    It remains to show item (1). By (4) the only thing that can go wrong is the regularity of vertices with color $\lambda(v_C)$. Similarly as we have shown item (6), this does not happen.
\end{proof}

\begin{notation}
    We denote the graph obtained in Proposition \ref{prop:removing-cliques} as $G\setminus \mathcal{C}$. The set of vertices in $C\in\mathcal{C}$ different from $v_C$ we denote by $V^\star(C)$. 
\end{notation}

\begin{definition}\label{def:depth-function}
    Let $G$ be a connected triangle-regular block graph $G$. We say that a function $\kappa_G: V(G)\cup E(G)\to \N$ is a \emph{depth function of $G$} if 
    \begin{enumerate}
        \item either $G$ has only one maximal clique and $\kappa
	\equiv 1$, or 
        \item $G$ has at least two maximal cliques and there is a
	collection  $\mathcal{C}$ as in Proposition
	\ref{prop:removing-cliques} such that the restriction
	$\kappa_{G \setminus \cC}$ of
	$\kappa_G$ to $V(G \setminus \cC) \cup E(G \setminus \cC)$ is
	a depth function for $G
	\setminus \cC$ and furthermore for any $C \in \cC$ and any
	$x \in (V(C) \setminus \{v_C\}) \cup E(C)$ 
	we have $\kappa_G(x)=\kappa_{G \setminus \cC}(v_C)+1$. 
    \end{enumerate}
\end{definition}

\begin{lemma}
    If $G$ is a connected triangle-regular block graph, then it has a depth function.
\end{lemma}
\begin{proof}
    This follows immediately from
    Proposition~\ref{prop:removing-cliques} and induction on the number of maximal cliques.
\end{proof}

\textbf{Convention:} Note that the depth function is not necessarily
unique. For example, let $G$ be the union of two distinct maximal cliques, then there exists two depth functions. 
In the following, let $G$ be a connected, triangle-regular block
graph, pick one of its depth functions, and call it $\kappa$. Whenever we use induction on the number of maximal cliques, we adopt the notation of Proposition \ref{prop:removing-cliques}.

\begin{lemma}\label{lemma:lambda-to-kappa}
    If $x,y\in G$ are edges or vertices with $\lambda(x) = \lambda(y)$, then $\kappa(x) = \kappa(y)$.
\end{lemma}
\begin{proof}
    We proceed by induction on the number of maximal cliques. If there
    is only one maximal clique, $\kappa$ is constant, hence the claim
    follows immediately. Otherwise, let $x$ and $y$ be vertices or
    edges such that $\lambda(x) = \lambda(y)$. If both are in
    $G\setminus \mathcal C$, $\kappa(x) = \kappa(y)$ by induction. If
    one of them is in $\mathcal C$, the other one is there too by item
    (3) of Proposition \ref{prop:removing-cliques}. Say $x\in C$ and
    $y\in D$ for some $C,D\in \mathcal C$. Then $\lambda(v_C) =
    \lambda(v_D)$ by item (4) of Proposition
    \ref{prop:removing-cliques}. By induction and the definition of a depth function, $\kappa(x) = \kappa(v_C) +1 = \kappa(v_D) +1 = \kappa(y)$.
\end{proof}

\begin{lemma}\label{lemma:kappa-edge-step}
    For all edges $e =\{u,v\}\in E(G)$, either $\lbrace \kappa(u),\kappa(v)\rbrace = \lbrace \kappa (e)\rbrace$ or $\lbrace \kappa(u),\kappa(v)\rbrace = \lbrace \kappa (e), \kappa (e)-1\rbrace$.
\end{lemma}
\begin{proof}
    We proceed by induction on the number of maximal cliques. If there
    is only one maximal clique, then $\kappa$ is constant, hence the
    claim follows immediately. Otherwise, if $e\in G\setminus\mathcal
    C$, then the claim follows by induction. If $e\in E(C)$ for some
    $C\in \mathcal C$, either $e$ does not contain $v_C$, and then we
    are in the first case, or it contains $v_C$, and then we are in
    the second case. 
\end{proof}

\begin{lemma}\label{lemma:up-and-down}
    For vertices $u,v\in V(G)$, let $(u= u_0,e_1,u_1,\dots ,e_n,u_n=v)$ be the shortest path. Then there exist a unique $m\in \lbrace0,1,\dots, n\rbrace$ such that
    \begin{itemize}
        \item $\kappa(u_{i+1}) = \kappa(u_i)-1$ for all $i< m$,
        \item $\kappa(u_{m+1}) = \kappa(u_m)$ or $\kappa(u_{m+1}) = \kappa(u_m) +1$,
        \item $\kappa(u_{i+1}) = \kappa(u_i)+1$ for all $i> m$,
    \end{itemize}
\end{lemma}
\begin{proof}
    We proceed by induction on the number of maximal cliques. If $G$ is just one maximal clique, $\kappa$ is constant and shortest paths are formed only by one edge or one vertex. Hence, the claim follows immediately.
    
    If $u$ and $v$ are in $G\setminus\mathcal C$, the claim follows by
    induction. If they are in the same clique $C\in \mathcal C$ and
    not in $G\setminus C$, the claim is analogous to the base case.
    Assume $u\in V^\star (C)$ and $v\in V^\star(D)$ for some
    $C,D\in\mathcal C$, $C\neq D$. By induction, the statement is true
    for $v_C$ and $v_D$. By Lemma \ref{lemma:path-clique-edge}, $v_C =
    u_1$ and $v_D = u_{n-1}$. Since $\kappa(u) = \kappa(v_C)+1$ and
    $\kappa(v) = \kappa(v_D)+1$, the claim follows. If $u$ is in
    $G\setminus\mathcal C$ and $v\in V^\star(D)$ for some
    $D\in\mathcal C$, we apply induction to $u$ and $v_D$. Then again
    $v_D = u_{n-1}$ and $\kappa(v) = \kappa(v_D)+1$. A similar
    argument applies when $u \in V^\star(C)$ for some $C \in \cC$ and
    $v$ is in $G \setminus \cC$.
\end{proof}

\begin{definition}
    We say that $\alpha:G\to G$ is a \emph{quasi-automorphism} if it is an automorphism of the underlying uncolored graph and for each vertex $v\in G$, $\lambda(\alpha(v)) = \lambda(v)$.
\end{definition}

In other words, a quasi-automorphism is an automorphism that respects
the vertex coloring but may ignore the edge coloring.  In Lemma
\ref{lemma:moving-paths}, however, we will show that there are
quasi-automorphisms that preserve the color of certain specific edges.

\begin{lemma}\label{lemma:quasi-automorphism}
    Consider $G$ and $\mathcal{C}$ as in Proposition \ref{prop:removing-cliques} and a quasi-automorphism $\Tilde{\alpha}$ of $G\setminus\mathcal{C}$. Then there exist a quasi-automorphism $\alpha$ of $G$ such that $\alpha\vert_{G\setminus\mathcal{C}} = \Tilde{\alpha}$.
\end{lemma}
\begin{proof}
    Consider a maximal set of vertices $U\subset \bigcup_{C \in \cC}
    V^\star(C)$ such that there is quasi-automorphism $\alpha$ of the
    induced graph on $V(G\setminus \mathcal{C})\cup U$ extending
    $\Tilde{\alpha}$. We show that $U = \bigcup _{C \in \cC}
    V^\star(C)$. 

    For a contradiction, assume there is a vertex $v\in
    V^\star(C)\setminus U$ for some $C \in \cC$. We consider two cases.

    Assume that $V(C)\cap U = \lbrace v_C\rbrace$. We know that
    $\lambda(\alpha(v_C)) = \lambda(v_C)$ and $\alpha$ is an
    automorphism of the underlying uncolored graph of $G \setminus
    \cC$. So by item (6) of Proposition \ref{prop:removing-cliques},
    there must be a clique $D\in \mathcal C$ containing $\alpha(v_C)$
    such that $\alpha(U)\cap V(D) = \lbrace v_D\rbrace$. By item (2)
    of Proposition \ref{prop:removing-cliques}, there is a vertex $v'
    \in V^\star(D)$ of color $\lambda(v)$, so we can extend $\alpha$ to $\alpha'$ by setting $\alpha'(v) = v'$.

    Assume that $V(C)\cap U \supsetneq \lbrace v_C\rbrace$. Then
    $\alpha$ maps $V(C)\cap U$ into a fixed clique $D\in \mathcal{C}$. By
    item (2) of Proposition \ref{prop:removing-cliques}, there is a
    vertex $v' \in V^\star(D)$ such that $\lambda(v') = \lambda(v)$ and $v'\not\in \text{Im}\,\alpha$. So we can extend $\alpha$ to $\alpha'$ by setting $\alpha'(v) = v'$.
\end{proof}

\begin{lemma}\label{lemma:local-q-a}
    Assume that we are in the situation of Proposition
    \ref{prop:removing-cliques}. Let $C,D \in \cC$ with $v_C=v_D$.
    Then for any distinct $u_1,\dots u_k\in V(C)$ and distinct
    $v_1,\dots v_k\in V(D)$ such that $\lambda(u_i) = \lambda(v_i)$ for all
    $i\in \lbrace 1,\dots ,k\rbrace$, there is a
    quasi-automorphism $\beta$ of $G$ such that
    $\beta\vert_{G\setminus (C\cup D)} = \text{id}$ and $\beta(u_i) =
    v_i$ for all $i\in \lbrace 1,\dots ,k\rbrace$. In particular,
    there is a quasi-automorphism sending an edge $e$ in $C$
    to any prescribed edge of color $\lambda(e)$ in 
    $D$.
\end{lemma}
\begin{proof}
    By item (2) of Proposition \ref{prop:removing-cliques}, $C$ and $D$
    have the same number of vertices and the same multiset
    of colors of vertices. Moreover, if some $u_i$ is equal to $v_C$,
    then since $D$ contains only one vertex of color $\lambda(v_C)$,
    also the corresponding $v_i$ is equal to $v_C$. Since $C$ and $D$
    are cliques, any map that is the identity on $G \setminus (C \cup
    D)$ and that sends $u_i \mapsto v_i$
    and preserves vertex colors is indeed a quasi-automorphism. 

    The second statement follows, using edge regularity, from the
    first statement with $k=2$.
\end{proof}

Note that the cliques $C$ and $D$ in the previous lemma are not required to be distinct.

In what follows, when we say that two paths $P,Q$ in a colored graph $G$ are
isomorphic, then we mean that there is an isomorphism between $P$ and
$Q$ that preserves vertex colors {\em and} edge colors.

\begin{lemma}\label{lemma:moving-paths}
    In a connected, triangle-regular block graph $G$, let $P$ and $Q$ be
    any two isomorphic shortest paths, and let $P'$ be a shortest path
     containing $P$. Then there exists a quasi-automorphism
    $\alpha$ of $G$ such that $\alpha(P) = Q$ (and indeed, $\alpha$ can
    be chosen to agree with any of the isomorphisms between $P$ and $Q$,
    of which there may be two) and $\lambda(\alpha(e)) = \lambda(e)$
    for all edges $e$ in $P'$.
\end{lemma}

\begin{proof}
    We proceed by induction on the number of maximal cliques; and we
    will write $\Tilde{P},\Tilde{Q},\Tilde{P}'$ for the paths to which
    we apply the induction hypothesis. There are three cases.

First, assume that $P$ is contained in a single $C \in
\mathcal{C}$. Then it has length $\leq 1$ by
Lemma~\ref{lemma:path-clique-edge}. 
By item (3) of Proposition \ref{prop:removing-cliques},
also $Q$ is contained in $D$ for some clique $D \in
\cC$. We apply induction to $\Tilde{P} = (v_C)$, $\Tilde{Q}
= (v_D)$ and $\Tilde{P'} = P'\cap (G\setminus \mathcal{C})$
to find a quasi-automorphism $\Tilde{\alpha}$ of $G \setminus
\mathcal{C}$ with $\Tilde{\alpha}(v_C)=v_D$ and
$\lambda(\Tilde{\alpha}(e))=\lambda(e)$ for all $e \in E(\Tilde{P}')$.
Then we apply Lemma \ref{lemma:quasi-automorphism} to extend
$\Tilde{\alpha}$ to a quasi-automorphism $\alpha$ of $G$ that maps $v_C$
to $v_D$. Further, applying Lemma~\ref{lemma:local-q-a} to enumerations of
$V^\star(\alpha(C))$ and $V^\star(D)$ starting with enumerations of the
vertices of $\alpha(P)$ and of $Q$, we find a quasi-automorphism $\beta$
such that $\gamma:=\beta \circ \alpha$ maps $C$ onto $D$ and $P$ onto
$Q$. This still satisfies $\lambda(\gamma(e))=\lambda(e)$ for all edges of
$P'$ in $G \setminus \cC$, but {\em a priori} it might not hold for the
remaining edges of $P'$. However, by Lemma~\ref{lemma:path-clique-edge},
any such remaining edge is of the form $\{v_{C'},v\}$ for some $C'
\in \cC$ and some $v \in V^\star(C')$, and its color is preserved by
$\gamma$ by item (5) of Proposition~\ref{prop:removing-cliques} and
the fact that $\gamma$ preserves vertex colors.

Second, assume that $P$ shares no edges with $G \setminus
\cC$ but is not contained in a single $C \in \cC$. Then using
Lemma \ref{lemma:path-clique-edge}, $P$ is of the form $P =
(v_0,e_1,v_C,e_2,v_2)$, where $e_1 \in E(C_1)$ and $e_2\in E(C_2)$ for
two maximal cliques $C_1,C_2 \in \mathcal{C}$ with $v_{C_1}=v_{C_2}$. Py
Proposition \ref{prop:removing-cliques}, the same is true for $Q$, so $Q=
(u_0,f_1,v_D,f_2,u_2)$. Note that $P$ is not contained in a strictly larger
shortest path, so $P'=P$. Using the induction hypothesis, there exists a
quasi-automorphism $\Tilde{\alpha}$ of $G \setminus \cC$ sending $v_C$ to
$v_D$. We can now extend this using Lemmas~\ref{lemma:quasi-automorphism}
and \ref{lemma:local-q-a} to send $P$ to $Q$.

Third, assume that $P$ does share edges with $G \setminus \cC$.  By item
(3) of Proposition \ref{prop:removing-cliques}, the same is true for $Q$,
and indeed $\Tilde{P} := P\cap (G\setminus\mathcal{C})$ and $\Tilde{Q}
:= Q\cap (G\setminus\mathcal{C})$ are isomorphic. We now apply the
induction hypothesis to $\Tilde{P}$ and $\Tilde{Q}$ and $\Tilde{P}':=
P' \cap (G \setminus \mathcal{C})$. The rest of the argument is similar
to the previous cases. 
\end{proof}

\begin{lemma}\label{lemma:symmetry}
    For any vertices $u,v\in G$ of the same color, $u\leftrightarrow
    v$ is palindromic.
\end{lemma}
\begin{proof}
    We proceed by induction on the number of maximal cliques. If $G$
    is only one clique, then shortest paths consist only of one vertex
    or one edge, so the claim is clear.

    If $u\in V(G\setminus\mathcal C)$, then by item (3) of Proposition
    \ref{prop:removing-cliques}, also $v\in V(G\setminus\mathcal C)$.
    Then the claim follows by induction. If $u$ and $v$ are in the
    same clique of $\mathcal{C}$, the claim is again immediate. So we
    can assume $u\in C$ and $v\in D$ for some $C,D\in\mathcal C$,
    $C\neq D$, and $u,v\not\in V(G\setminus \mathcal C)$. By item (4)
    of Proposition \ref{prop:removing-cliques}, $\lambda(v_C) =
    \lambda(v_D)$, so we can apply induction to $v_C$ and $v_D$. The
    path $v_C\leftrightarrow v_D$ is a subpath of $u\leftrightarrow v$
    and $\lambda(\{u,v_C\}) = \lambda(\{v,v_D\})$ by (5) in Proposition \ref{prop:removing-cliques}, so the path
    $u\leftrightarrow v$ is indeed palindromic.
\end{proof}

We denote by $ \Lambda (u \leftrightarrow v)$ the multiset of edge and vertex colors contained in the shortest path $u\leftrightarrow v$. We say that two shortest paths $u\leftrightarrow v$ and $u'\leftrightarrow v'$ are \emph{combinatorially equivalent}, if $ \Lambda (u \leftrightarrow v) = \Lambda (u' \leftrightarrow v')$. 

\begin{lemma}\label{lemma:comb-equiv-to-iso}
    Any two combinatorially equivalent shortest paths in a triangle-regular block graph are isomorphic.
\end{lemma}
\begin{proof}
    We proceed by induction on the number of maximal cliques. Let $u,v,u',v'\in V(G)$ be such that the paths $u\leftrightarrow v$ and $u'\leftrightarrow v'$ are combinatorially equivalent.

    The only non-trivial case is when $u\in V^\star(C)$ and $v\in
    V^\star(D)$ for some $C,D\in\mathcal C$, $C\neq D$. The path $u
    \leftrightarrow v$ does
    not contain any other vertices from $V^\star(C')$ for 
    $C'\in\mathcal C$, otherwise it would violate Lemma
    \ref{lemma:path-clique-edge} for $C'$. By item (3) of Proposition
    \ref{prop:removing-cliques}, the path $u\leftrightarrow v$
    contains exactly two vertices of colors appearing in
    $\lambda(V^\star (C))=\lambda(V^\star(D))$, namely $u$ and $v$. Since
    $u\leftrightarrow v$ is combinatorially equivalent to
    $u'\leftrightarrow v'$, the same must be true for
    $u'\leftrightarrow v'$. Moreover, those two vertices have to be
    $u'$ and $v'$, otherwise Lemma \ref{lemma:path-clique-edge} is
    violated. So $u'\in V^\star (C')$ and $v'\in V^\star (D')$ for
    some $C',D'\in \mathcal C$. Then we apply the induction hypothesis to
    $v_C\leftrightarrow v_D$ and $v_{C'}\leftrightarrow v_{D'}$. By
    item (4) of Proposition \ref{prop:removing-cliques}, $\lambda(v_C)
    = \lambda(v_D) = \lambda(v_{C'}) = \lambda(v_{D'})$. Applying
    Lemma \ref{lemma:symmetry}, we can assume that $\lambda(u) =
    \lambda(u')$ and $\lambda(v) = \lambda(v')$. Then by item (5) of
    Proposition \ref{prop:removing-cliques}, $\lambda(\{u,v_C\}) =
    \lambda(\{u',v_{C'}\})$ and $\lambda(\{v,v_D\}) =
    \lambda(\{v',v_{D'}\})$, so the claim follows.
\end{proof}

\section{Generators of the vanishing ideal} \label{sec:Gens}

The main idea of this section is to show that the vanishing ideal equals
the kernel of the map $\psi_G$ introduced in \cite{CMMPS23} for RCOP
block graphs. Throughout the section, we use the weaker assumption that
$G$ is a connected, triangle-regular block graph.

\begin{definition}
    The map $\psi_G$ is defined by 
    \begin{align*}
        \psi_G: \R[\Sigma] &\to \R[t_\ell\mid \ell \in \lambda(V(G)
	\cup E(G))]\\
        \sigma_{i,j}&\mapsto \prod_{v\in V(i\leftrightarrow
	j)}t_{\lambda(v)}  \prod_{e\in E(i\leftrightarrow j)}
	t_{\lambda(e)}.
    \end{align*}
\end{definition}

The kernel of this map is a binomial ideal. 

\subsection{Completion of a triangle-regular block graph}
We define the completion of a triangle-regular block graph the same way as \cite{CMMPS23} defined the completion of an RCOP block graph.
\begin{definition}
    Let $G$ be a triangle-regular block graph. The graph has vertex set $V$, edge set $E$ and coloring  $\lambda $. The  \emph{completion} of $G$, denoted by $\hat G$, is the complete graph on $V$ with the coloring $\hat{\lambda}$ defined by:
    \begin{itemize}
        \item $\hat{\lambda } (v) = \lambda (v) $ and $\hat{\lambda} (e) = \lambda (e) $ for all $v \in V$
        and $e \in E.$
        \item Edges $\{u,v\}$ that are not in $G$ have new colors assigned to them such that for \newline
        $\{u,v\}, \{u',v'\} \not \in E$ it holds $\hat{\lambda } (\{u,v\}) = \hat{\lambda } (\{u',v'\})$ iff $\Lambda (u \leftrightarrow v) = \Lambda (u'\leftrightarrow v').$ 
    \end{itemize}
\end{definition}

The following proposition is essentially similar to Lemma \ref{lemma:moving-paths}.

\begin{proposition}
    The completion of a triangle-regular block graph is triangle-regular.
\end{proposition}
\begin{proof}
    \textit{Edge regularity:} This follows immediately from Lemma \ref{lemma:comb-equiv-to-iso}.
    
    \textit{Vertex regularity:} We prove the following slightly
    stronger statement: For any $u,v\in V(G)$ such that $\lambda(u) =
    \lambda(v)$, there is a quasi-automorphism $\alpha$ of $G$ such that for any vertex $x\in V(G)$, $\Lambda(u\leftrightarrow x) = \Lambda(v\leftrightarrow \alpha (x))$. We proceed by induction on the number of maximal cliques. 

    The base case is when the graph is a clique. Then it is already
    its own completion, and we define $\alpha$ by enumerating $V(G)$
    as $u=u_1,u_2,\ldots,u_k$ and as $v=v_1,v_2,\ldots,v_k$ in such a
    manner that $\lambda(\{u,u_i\})=\lambda(\{v,v_i\})$ for all $i$,
    and setting $\alpha(u_i):=v_i$. 

    If $G$ is not a clique, assume first that $u\in V^\star(C)$ for some $C
    \in \cC$. Then also $v\in V^\star(D)$ for some $D \in \cC$ by
    item (3) of Proposition \ref{prop:removing-cliques}. By
    induction, there exists a quasi-automorphism $\Tilde{\alpha}$ of
    $G\setminus\mathcal C$ such that for every $x\in V(G\setminus
    \mathcal{C})$, $\Lambda(v_C\leftrightarrow x) =
    \Lambda(v_D\leftrightarrow x)$. We extend $\Tilde{\alpha}$ by
    Lemma \ref{lemma:quasi-automorphism} to a quasi-automorphism
    $\alpha$ of $G$. Clearly, $\alpha(v_C) = v_D$. By Lemma
    \ref{lemma:local-q-a} and edge regularity, there exists
    quasi-automorphism $\beta$ such that $\beta\circ\alpha(u) = v$ and
    for every $x\in C$, $\lambda(\{v,\beta\circ\alpha(x)\}) =
    \lambda(\{u,x\})$. We claim that $\alpha\circ \beta$ is the desired quasi-automorphism: For $x\in V(C)$, the claim is true by construction. For $x\in V(G\setminus \mathcal{C})$, the claim is true by induction and item (5) of Proposition \ref{prop:removing-cliques}. For $x\in V(\mathcal C\setminus C)$, the claim is true by induction and item (5) of Proposition \ref{prop:removing-cliques} used twice.

    If $u$ is not in any $\bigcup_{C \in \mathcal C} V^\star(C)$, then neither
    is $v$. By induction, there is a quasi-automorphism $\Tilde{\alpha}$
    of $G \setminus \cC$ with the property that $\Lambda(u \leftrightarrow
    x)=\Lambda(v \leftrightarrow \alpha(x))$ for all $x \in V(G \setminus
    \cC)$. Extend $\Tilde{\alpha}$ to a quasi-automorphism $\alpha$ of
    $G$ using Lemma~\ref{lemma:quasi-automorphism}. This has the desired
    property, by the property of $\Tilde{\alpha}$ on the one hand and by
    item (5) of Proposition~\ref{prop:removing-cliques} on the other hand.

    \textit{Edge triangle regularity:} We proceed similarly to the vertex
    regularity. Let $(u_1,u_2)$ and $(v_1,v_2)$ be pairs of vertices such
    that $\Lambda(u_1\leftrightarrow u_2) = \Lambda(v_1\leftrightarrow
    v_2)$. Then the following cases need to be distinguished in the
    induction step: $u_1,u_2$ are both in $V(G \setminus \cC)$; exactly
    one of them is; $u_i \in V^\star(C_i)$ for $i=1,2$ with $C_1,C_2 \in
    \cC$ distinct; and $u_1,u_2 \in V^\star(C)$ for some $C \in \cC$. No
    new ideas are needed, except the use of triangle regularity of $C$
    in the last case.
\end{proof}

The following proposition uses the same idea as Theorem 5.5 in \cite{CMMPS23}. We see that the graph being triangle-regular is exactly the condition needed for their argument, and we do not need the existence of an automorphism in this proof. 

\begin{proposition}
\label{prop:completeGraphIsJordanAlgebra}
Let $\mathcal{G}$ be a complete triangle-regular block graph. Then
$\mathcal{L}_{\mathcal{G}}$ is a Jordan subalgebra of the space of
symmetric $V(G) \times V(G)$-matrices. 
\end{proposition}
\begin{proof}
To show that $\mathcal{L}_{\mathcal{G}}$ is a Jordan subalgebra, it is
enough to show that given a matrix $K \in \mathcal{L}_{\mathcal{G}}$,
$K ^ 2$ is also in $\mathcal{L}_{\mathcal{G}}$. So we have to show
that for vertices $m,m'$ of the same color, we have $(K^2)_{mm} =
(K^2)_{m'm'}$ and for edges $\{i,j\},\{i',j'\}$ of the same color,
we have $(K^2)_{ij} = (K^2)_{i'j'}$.

For entries on the diagonal we have
\begin{align*}
(K^2)_{mm} = \sum _{l = 1}^n k_{ml}k_{lm} = (k_{mm})^2 + \sum
_{\{l,m\} \in E(G)} (k_{ml})^2 
\end{align*}
If $m,m'$ have the same color, then $k_{mm} = k_{m'm'}$. Similarly,
the term $k_{ml}$ only depends on the color of the edge $\{m,l\}$. Now
if $m,m'$ have the same color, using vertex regularity, they are adjacent to the same multiset
of edge colors, hence the sums $\sum _{l \neq m} (k_{ml})^2 = \sum _{l
\neq m'} (k_{m'l})^2 $ are equal.

For an edge $\{i,j\}$, we have:
\begin{align*}
(K^2)_{ij} = \sum _{l = 1}^n K_{il} K_{yj} =  k_{ii}k_{ij} + k_{ij}k_{jj} + \sum _{l \neq i, l \neq j} k_{il}k_{lj}.
\end{align*}
The first two terms are equal to the corresponding terms for
an edge $\{i',j'\}$ of the same color as $\{i,j\}$ using edge regularity. 
A term $k_{il}k_{lj}$ corresponds to a triangle with edges of color
$\lambda (\{i,l\}),\lambda (\{l,j\})$. The edge
$\{i',j'\}$ has the same color as $\{i,j\}$, so using edge triangle regularity, they are incident to
the same multiset of colored triangles. Therefore also $\sum _{l \neq
i, l \neq j} k_{il}k_{lj} = \sum _{l \neq i', l \neq j'}
k_{i'l}k_{lj'}$ and we find $ (K^2)_{ij} = (K^2)_{i'j'}$.
\end{proof}

Using the following theorem, we see that for complete triangle-regular graphs, $\mathcal{L}_G = \mathcal{L}^{-1}_G$.

\begin{theorem}{\cite[Lemma~1]{Jensen88}}
    Let $\mathcal{L}$ be a linear space of symmetric matrices that
    intersects the set of invertible matrices. Let $\mathcal{L}^{-1}$
    denote the Zariski closure of the set of all inverses of matrices
    in $\mathcal{L}$. Then $\mathcal{L}$ is a Jordan subalgebra of the
    Jordan algebra of symmetric matrices if and only if $\mathcal{L} = \mathcal{L}^{-1}$.
\end{theorem}
Therefore, we can conclude that the equations satisfying $I_{\hat{G}}$ are the ones satisfied by $K \in \mathcal{L}_{\hat{G}}$:
\begin{corollary} \label{cor:completion-eqs}
    Let $\hat{G}$ be the completion of a triangle-regular block graph
    $G$. Then the binomial $I_{\hat{G}}$ is generated by the linear
    binomials
    \[
    I_{\hat{G}} =\langle \sigma _{ij} - \sigma _{kl} \mid \Lambda (i \leftrightarrow j) = \Lambda (k \leftrightarrow l) \text{ in } G \rangle.
    \]
\end{corollary}
We conclude that complete triangle-regular block graphs have binomial vanishing ideals. 

\subsection{Proof of the generators in Theorem~\ref{thm:CharToric}.}

Whenever $G$ is not a single clique, we will use the setting of
Proposition~\ref{prop:removing-cliques} and a depth function $\kappa$ as in
Definition~\ref{def:depth-function}.

\begin{definition}
Let $P,Q$ be shortest paths in $G$ that meet at a vertex $x$. Let
$P_1,P_2$ be the unique paths that both have $x$ as an endpoint,
satisfy $E(P)=E(P_1) \sqcup E(P_2)$ and where $P_1$ either has
length $0$ or else contains at least one vertex $x' \neq x$ with
$\kappa(x') \leq \kappa(x)$. We use Lemma~\ref{lemma:up-and-down}
to see that this decomposition exists and is unique; it has the
property that $P_2=(x=x_0,e_1,x_1,\ldots,e_k,x_k)$ with $k \geq 0$ and
$\kappa(x_{i+1})=\kappa(x_i)+1$ for all $i \in \{0,\ldots,k-1\}$. Decompose $Q$ in the same
manner into $Q_1$ and $Q_2$. It then follows from Corollary \ref{Cor:shortest-path-criterion} that the concatenation
$(P_1,Q_2)$ of $P_1$ and $Q_2$ is also a shortest path, and so is the
concatenation $(Q_1,P_2)$. We call $((P_1,Q_2),(Q_1,P_2))$ the {\em swap}
of $(P,Q)$ at $x$.
\end{definition}

\begin{theorem}\label{thm:sets-of-paths}
    Let $\mathcal A$ and $\mathcal B$ be two multisets of shortest paths in a triangle-regular block graph $G$ such that the multiset of colors of edges and the multiset of colors of vertices in $\mathcal{A}$ are the same as those in $\mathcal B$. We are allowed to modify those two sets by the following two moves:
    \begin{enumerate}
        \item Replace a path in $\cA$ by an isomorphic shortest path.
        \item Replace a pair $(P,Q)$ of shortest paths in $\cA$ that
	meet at some vertex $x$ by the swap $((P_1,Q_2),(Q_1,P_2))$ of
	$(P,Q)$ at $x$.
    \end{enumerate}
    Then there exists a finite sequence of moves which achieves $\mathcal A=\mathcal B$.
\end{theorem}

\begin{proof}
We proceed by a double induction: on the cardinality of $\cA$, and then
on the sum of the lengths of the paths in $\cA$.  If all paths in $\cA$
have length zero, then the multiset of edge colors of paths in $\cA$
is empty and hence the same holds for $\cB$. In this case, the theorem
clearly holds. 

Otherwise, let $k$ be the maximum depth among all edges appearing in
paths of $\cA$. Such an edge either has two vertices of depth $k$ or
else one vertex of depth $k$ and one vertex of depth $k-1$.

If the former case applies to some $e$, then any path $P$ in
$\cA$ containing a copy of $e$ consists of $e$ and its vertices
only. By the assumptions on $\cA$ and $\cB$, $\cB$ has a path
$P'$ that contains an edge $e'$ with $\lambda(e')=\lambda(e)$. By
Proposition~\ref{prop:removing-cliques}, $P'$ equals $(u',f',v')$ and
is therefore isomorphic to $P$. The theorem now reduces to the smaller
instance with $\cA \setminus \{P\}$ and $\cB \setminus \{P'\}$, and we
are done by induction.

Suppose that there strictly more vertices of depth $k$ in paths in $\cA$
than there are edges of this depth. Then some path $P$ in $\cA$ consists
of a single depth-$k$ vertex $u$. The assumptions on $\cB$ and $\cA$
imply that among the vertices in paths in $\cB$ there are more of color
$\lambda(u)$ than there are edges in paths in $\cB$ with an endpoint of
color $\lambda(u)$. Hence $\cB$, too, has a path $P'$ consisting of a
single vertex $u'$ with $\lambda(u)=\lambda(u')$. Removing $P$ from $\cA$
and the isomorphic path $P'$ from $\cB$, we are done by induction.

So we may assume that all edges of depth $k$ in paths in $\cA$ connect
a vertex of depth $k$ to a vertex of depth $k-1$, and that all paths in
$\cA$ with a vertex $u$ of depth $k$ have $u$ as endpoint and contain
at least one edge, connecting $u$ to a vertex of depth $k-1$. The same
statements then apply to $\cB$. 

Pick an edge $e$ of depth $k$ in a path $P$ in $\cA$, and let $u,v$
be its vertices of depth $k,k-1$, respectively. Then $\cB$ contains
a path $P'$ with an edge $e'$ with $\lambda(e')=\lambda(e)$, and it
follows that $\kappa(e')=k$, the vertex $u'$ of $e'$ of depth $k$ is
the endpoint of $P'$, and the other vertex $v'$ of $e'$ has depth $k-1$.

Now remove $e$ and $u$ from $P$ to obtain $\Tilde{P}$ and $e$ and $u'$
from $P'$ to obtain $\Tilde{P}'$, and let $\Tilde{\cA}$ and $\Tilde{\cB}$
be the shortest path collections obtained by replacing $P$ in $\cA$
by $\Tilde{P}$ and $P'$ in $\cB$ by $\Tilde{P}'$. As the sum of the
lengths of paths in $\cA$ has shrunk by $1$, the induction
hypothesis says that we can make $\Tilde{\cA}$ equal to $\Tilde{\cB}$ by a sequence
of moves.

We carry out the same moves for $\cA$ in the following sense. In moves
of type (1), whenever $\Tilde{P}$ is replaced by an isomorphic path
$\Tilde{Q}$, by Lemma~\ref{lemma:moving-paths} there is a shortest path
$Q$ in $G$ isomorphic to $P$ containing $\Tilde{Q}$, and we replace $P$
in $\cA$ by $Q$. In the process we update $e$ and $u$ and $v$
accordingly, so that $P$ still starts with $(u,e,v)$. And in a move of type (2)
that involves $\Tilde{P}$ and another path $Q$, say with $v$ an endpoint of $\Tilde{P}_i$, we update $\Tilde{P}$ to the new path containing $\Tilde{P}_i$,
so that $\tilde{P}$ again has $v$ as one endpoint, and in $\cA$ we 
replace $(P,Q)$ by the swap of $(P,Q)$ at $x$; then the new $P$ is the
new $\tilde{P}$ with $e$ and $u$ attached. 

After these moves, $\cA$ and $\cB$ become equal as multisets when
we remove $u,e$ from $P$ and $u',e'$ from $P'$; we continue to write
$\Tilde{P}$ and $\Tilde{P}'$, respectively, for these truncated paths.
If $\Tilde{P}=\Tilde{P}'$, then $P$ and $P'$ are isomorphic (this needs
Lemma~\ref{lemma:symmetry} when both endpoints of $\Tilde{P}$ have the
color $\lambda(v)$). In that case, we can remove $P$ from $\cA$ and $P'$
from $\cB$, and we are done by induction.  Otherwise, let $Q \in \cA$
be an element equal to $\Tilde{P}'$. In particular, $Q$ has $v'$ as an
endpoint.  Replace $P \in \cA$ by an isomorphic path that contains
$u',e',v'$. Then replace $(P,Q)$ in $A$ by its swap at $v'$; this has
the effect that $u'$ and $e'$ are removed from $P$ and attached to $Q$,
so that $Q$ becomes equal to $P'$. After removing $Q$ from $\cA$ and $P'$
from $\cB$, we are done by induction.
\end{proof}

\begin{theorem}{\cite[Theorem~1 \& Theorem~5]{MS21}}\label{thm:block-graph-generators}
    Let $G_0$ be a connected uncolored block graph. Then 
    $$
        I_{G_0} = \langle \sigma_{i,j}\sigma_{k,l}-\sigma_{i,k}\sigma_{j,l}\mid E(i\leftrightarrow j)\cup  E(k\leftrightarrow l)= E(i\leftrightarrow k)\cup  E(j\leftrightarrow l)\rangle
    $$
\end{theorem}   
\begin{corollary}
\label{cor:kerPsiGGenerators}
    The binomial ideal $\ker (\psi _G)$ is generated by $I_{\hat G}$ and $I_{G_0}$, where $\hat G$ is the completion of $G$ and $G_0$ is the underlying uncolored graph of $G$.
\end{corollary}
\begin{proof}
     The kernel $\ker(\psi_G)$ is generated by binomials of the form 
     \[\sigma_{i_1,j_1}\dots \sigma_{i_m,j_m} - \sigma_{k_1,l_1}\dots \sigma_{k_m,l_m}.\]
     Using the definition of $\psi _G$, the multiset of colors of all edges and vertices along each shortest path $i_1\leftrightarrow j_1$, $\dots$, $i_m\leftrightarrow j_m$ is the same as the one of $k_1\leftrightarrow l_1$, $\dots$, $k_m\leftrightarrow l_m$. Let $\mathcal{A} = \lbrace i_1\leftrightarrow j_1,\dots,i_m\leftrightarrow j_m \rbrace$ and $\mathcal{B} = \lbrace k_1\leftrightarrow l_1,\dots,k_m\leftrightarrow l_m \rbrace$. Applying a generator of $I_{\hat G}$ corresponds to a replacement of type (1) in Theorem \ref{thm:sets-of-paths}. A replacement of type (2) corresponds to a generator of $I_{G_0}$. Hence, Theorem \ref{thm:sets-of-paths} completes the proof.
\end{proof}

To finish, we follow the same argument as in \cite{CMMPS23}. 

\begin{corollary}{\cite[Theorem~7.7]{CMMPS23}}
    Let $G$ be a triangle-regular block graph. Then  $I_G$ is the union of $I_{\hat G}$ and $I_{G_0}$, where $\hat G$ is the completion of $G$ and $G_0$ is the underlying uncolored graph of $G$.
\end{corollary}
\begin{proof}
Using Corollary~\ref{cor:kerPsiGGenerators}, we are left to show that $I_G = \ker(\psi _G)$. The linear space $\mathcal{L}_G$ can be constructed from $\mathcal{L}_{\hat G}$ by coloring the graph and from $\mathcal{L}_{G_0}$ by removing edges, so $\mathcal{L}_G \subset \mathcal{L}_{\hat G}, \mathcal{L}_{G_0}$. Therefor $I_{G_0}, I_{\hat G} \subset I_G$, and hence $\ker (\psi _G) = I_{G_0} + I_{\hat G} \subset I_G $. Both ideals are prime, so it suffices to show that they have the same dimension to conclude that they are equal. We have $\dim (\ker (\psi _G)) \geq \dim (I_G)$. The number of colors is an upper bound for the dimensions of $I_G$ and $\ker(\psi _G)$, and this upper bound is attained in $I_G$. Hence the dimensions are equal.
\end{proof}

\bibliographystyle{alpha}\bibliography{toricbib}

\end{document}